\documentclass[12pt,a4paper]{amsart}

\usepackage{amssymb}
\usepackage{amscd}

\def\frak{\mathfrak}
\def\Bbb{\mathbb}
\def\Cal{\mathcal}

\let\phi\varphi
\let\ocirc\circledcirc

\renewcommand{\Re}{\operatorname{Re}}

\newcommand{\x}{\times}
\renewcommand{\o}{\circ}

\newcommand{\al}{\alpha}
\newcommand{\be}{\beta}

\newcommand{\ep}{\epsilon}

\newcommand{\la}{\lambda}
\newcommand{\om}{\omega}
\newcommand{\ph}{\phi}
\newcommand{\ps}{\psi}
\renewcommand{\th}{\theta}
\newcommand{\si}{\sigma}

\newcommand{\Ga}{\Gamma}
\newcommand{\La}{\Lambda}

\newcommand{\vol}{\operatorname{vol}}
\newcommand{\im}{\operatorname{im}}
\newcommand{\id}{\operatorname{id}}

\newcommand{\Ad}{\operatorname{Ad}}
\newcommand{\ad}{\operatorname{ad}}

\newcommand{\rpl}                         % +) or <+
{\mbox{$
\begin{picture}(12.7,8)(-.5,-1)
\put(0,0.2){$+$}
\put(4.2,2.8){\oval(8,8)[r]}
\end{picture}$}}

\numberwithin{equation}{section}

\newcounter{theorem}
\numberwithin{theorem}{section}
\newtheorem{thm}[theorem]{Theorem}
\newtheorem*{thm*}{Theorem \thesubsection}
\newtheorem{lemma}[theorem]{Lemma}
\newtheorem{prop}[theorem]{Proposition}
\newtheorem{cor}[theorem]{Corollary}
\newtheorem*{lemma*}{Lemma \thesubsection}
\newtheorem*{prop*}{Proposition \thesubsection}
\newtheorem*{cor*}{Corollary \thesubsection}

\theoremstyle{definition}
\newtheorem{definition}[theorem]{Definition}
\newtheorem*{definition*}{Definition \thesubsection}
\newtheorem{example}[theorem]{Example}
\newtheorem*{example*}{Example \thesubsection}
\theoremstyle{remark}
\newtheorem{remark}[theorem]{Remark}
\newtheorem*{remark*}{Remark \thesubsection}

\def\sideremark#1{\ifvmode\leavevmode\fi\vadjust{\vbox to0pt{\vss% the remark
 \hbox to 0pt{\hskip\hsize\hskip1em%                          will appear only
 \vbox{\hsize3cm\tiny\raggedright\pretolerance10000%          on the side
  \noindent #1\hfill}\hss}\vbox to8pt{\vfil}\vss}}}%
\begin{document}
\title{Parabolic Compactification\\ of Homogeneous Spaces}

\author{Andreas \v Cap, A.\ Rod Gover, and  Matthias Hammerl}

\address{A.\v C. \& M.H.: Faculty of Mathematics\\
University of Vienna\\
Oskar--Morgenstern--Platz 1\\
1090 Wien\\
Austria\\
A.R.G.:Department of Mathematics\\
  The University of Auckland\\
  Private Bag 92019\\
  Auckland 1142\\
  New Zealand}
\email{Andreas.Cap@univie.ac.at}
\email{r.gover@auckland.ac.nz}
\email{matthias.hammerl@univie.ac.at}

\begin{abstract}
In this article, we study compactifications of homogeneous spaces
coming from equivariant, open embeddings into a generalized flag
manifold $G/P$. The key to this approach is that in each case $G/P$ is the
homogeneous model for a  parabolic geometry; the theory of such geometries provides a
large supply of geometric tools and invariant differential operators
that can be used for this study. A classical theorem of J.~Wolf shows
that any involutive automorphism of a semisimple Lie group $G$ with
fixed point group $H$ gives rise to a large family of such
compactifications of homogeneous spaces of $H$. Most examples of
(classical) Riemannian symmetric spaces as well as many non--symmetric
examples arise in this way. A specific feature of the approach is that
any compactification of that type comes with the notion of ``curved
analog'' to which the tools we develop also apply. The model example
of this is a general Poincar\'e--Einstein manifold forming the curved analog
of the conformal compactification of hyperbolic space.

In the first part of the article, we derive general tools for the
analysis of such compactifications. In the second part, we analyze two
families of examples in detail, which in particular contain
compactifications of the symmetric spaces $SL(n,\Bbb R)/SO(p,n-p)$ and
$SO(n,\Bbb C)/SO(n)$. We describe the decomposition of the
compactification into orbits, show how orbit closures can be described
as the zero sets of smooth solutions to certain invariant differential
operators and prove a local slice theorem around each orbit in these
examples.
  \end{abstract}

\maketitle

\subjclass{}

\pagestyle{myheadings} \markboth{\v Cap, Gover, Hammerl}{Parabolic
  Compactifications} 

\thanks{}

\section{Introduction}\label{1}

To study non-compact manifolds it can be helpful to add a suitable
boundary structure so that the resulting space is compact. Within
geometry this idea is perhaps historically most well known in
hyperbolic geometry and \cite{Mo}, for example, provides a striking
application of this point of view. A special case that has received a
lot of attention in the literature is the question of compactifying
symmetric and locally symmetric spaces, see for example the monograph
\cite{BorelJi} and references therein. The prototype case of such a compactification is
provided by the Poinncar\'e ball compactification of real hyperbolic
space. This compactifies hyperbolic space by adding a sphere as a
boundary at infinity. While completeness of the hyperbolic metric
implies that it cannot be smoothly extended to the boundary, its
underlying conformal structure does admit a smooth extension, thus
endowing the boundary sphere with its standard conformal structure. Penrose's
concept of conformal compactness provides an analogous notion for more
general (pseudo--)Riemannian manifolds, with Poincar\'e--Einstein
manifolds forming an important special case. These ideas have
been been extremely fruitful with applications to topics like
negatively curved Riemannian manifolds, geometric scattering, general
relativity (GR), conformal geometry, and the AdS/CFT correspondence of
physics, see
\cite{Maldacena,FeffGr,Fr,GrZ,MazzeoMel,Mel,Penrose,Vasy}
but also in representation theory and harmonic analysis, see e.g.\
\cite{KOSS}.

During the last years a new conceptual approach to conformal
compactness and Poincar\'e--Einstein manifolds has been developed, see
\cite{G-sigma,Gover:P-E}. Rather than viewing the metric in the interior and
the conformal structure on the boundary as the basic objects, these
approaches are based on the conformal structure on a manifold with
boundary, together with a defining density for the boundary which
automatically selects a metric from the conformal class in the
interior. The advantage of this approach is that, using tools of
conformal geometry, it immediately leads to a host of geometric
objects that admit a smooth extension to the boundary, and indeed
beyond the boundary. These then provide powerful
tools for efficiently and systematically treating many of the problems
linked to the applications mentioned above
\cite{GLW-Mem,GW-boundarycalc,GW-volume}. This description
also leads to an interpretation of Poincar\'e--Einstein metrics as a
certain type of reduction of conformal holonomy. The general versions
of tractor calculus, see \cite{CG-TAMS}, and the theory of holonomy
reductions of Cartan geometries developed in \cite{hol-red} then
shows that many of these ideas can be extended from conformal geometry
to the class of Cartan geometries. The latter includes the rich class
of parabolic geometries \cite{book}.

In particular, analogs of the concept of conformal compactness in the
setting of projective and of c--projective differential geometry have
been introduced and studied in
\cite{Proj-comp,Proj-Comp2,CProj-Comp,ageom}. In all these cases, the
theory of parabolic geometries and the machinery of
Bernstein--Gelfand--Gelfand sequences introduced in \cite{CSS-BGG} and
\cite{Calderbank-Diemer} provide conceptual ways to obtain and
identify geometric quantities that automatically admit a smooth
extension to the boundary. This was a crucial input for the
developments in the articles referred to above. At the same time,
these procedures  produce families of  PDEs, both on the  in the
interior and on the boundary, that are naturally associated with the compactified
geometry.
Most critically some of these equations also come
with canonical solutions that combine with the underlying higher order
Cartan geometry to define the interior and boundary geometries, as well
providing a concrete link between these.

From the perspective of compactifying symmetric spaces, the example of
hyperbolic space is deceptively simple. In general one can definitely
not expect that it will be sufficient to just add, to a given non-compact
symmetric space, a boundary of codimension one at infinity. Rather one has to
expect a complicated family of boundary components of different
dimension, attached to each other in a highly complicated
fashion. Correspondingly, the compactifications are often mainly
understood from the point of view of topology.  Thus even getting to a
setting that allows for a notion of the geometry of such a compactification
is often very difficult.

The aim of the current paper is to apply the ideas on holonomy reductions of
parabolic geometries, on the level of their homogeneous models, to construct and study
compactifications of certain homogeneous spaces, among which there are many symmetric
spaces. The basic strategy is to consider a generalized flag variety $G/P$ of a
semisimple Lie group $G$ as well as a subgroup $H\subset G$, such that the obvious
action of $H$ on $G/P$ has at least one open orbit. A classical result of J.~Wolf
(see Theorem \ref{2.6}) implies that this is the case (for any choice of parabolic
subgroup $P\subset G$) if $H\subset G$ is the fixed point group of an involutive
automorphism of $G$. This immediately leads to a large number of interesting
examples. Choosing a point in an open orbit and denoting by $K$ the stabilizer of
that point in $H$, the orbit gets identified with the homogeneous space $H/K$. Since
$G/P$ is compact, the closure of the orbit forms a natural compactification
$\overline{H/K}$ of $H/K$. By construction $\overline{H/K}$ can be written as a union
of $H$--orbits, which automatically are initial submanifolds of $G/P$, thus providing
a decomposition of the boundary.
 
Now $G/P$ is the homogeneous model of parabolic geometries of type
$(G,P)$. Such geometries can be restricted to open subsets, so $H/K$
carries an $H$--invariant locally flat parabolic geometry of that
type. Moreover, by construction this geometry admits a smooth
extension across the boundary of $H/K$ in $G/P$, thus providing a
first set of geometric objects that extend smoothly. Moreover, any
$H$--invariant element in a representation $\Bbb V$ of $G$ defines a
parallel section in a tractor bundle naturally associated to the
parabolic geometry. Such a section can be projected to a section of
simpler bundle, which lies in the kernel of an overdetermined linear
differential operator (a so--called ``first BGG operator'') naturally
associated to the parabolic geometry. Both sections are defined on all
of $G/P$ and thus extend smoothly to (and across) the boundary and, by
$H$--invariance, they can be often used to distinguish between
different $H$--orbits in the boundary. The fact that one obtains
parallel sections of tractor bundles and sections in the kernel of a
first BGG operator, respectively, allows one to get good control on
the derivatives (and even on higher jets) of such BGG solutions based
on information coming from representation theory. We show that this
can be used to obtain detailed descriptions of the topology,
smooth structure, and geometry of the boundary structure.

We want to point out at this stage that, although in this article we
formally only consider homogeneous
spaces, the ideas introduced here automatically extend to more general
settings. On the one hand, there is the possibility to pass to certain
non-compact locally homogeneous spaces by
factoring by appropriate discrete subgroups of $H$. The model example
here is again provided by hyperbolic space, for which one may
factor by convex cocompact subgroups of $SO_0(n+1,1)$ and still
attach a boundary that locally is as in the model \cite{Ratcliffe}. On
the other hand, there is the possibility to replace the homogeneous
model $G/P$ by a curved parabolic geometry of type $(G,P)$ endowed
with a holonomy reduction (as a Cartan geometry) to the subgroup
$H\subset G/P$ in the sense of \cite{hol-red}. The theory developed in
that reference shows that the boundary structure in such a curved
holonomy reduction can be nicely compared to the situation on the
homogeneous model via the so--called curved orbit decomposition. From
the point of view of the ambient parabolic geometry, the existence of
such a holonomy reduction of course is a very restrictive condition,
but as examples like Poincar\'e Einstein manifolds, the K\"{a}hler
analogues of these, and their generalizations show, it is to be
expected that there are many interesting examples in the curved case.

Let us briefly outline how the article is organized. The general
theory of compactifications, as outlined above, is developed in
Section \ref{2}. We start with a general definition of homogeneous
compactifications and this already leads to first results on the
boundary structure. We then specialize to parabolic compactifications
defined by open $H$--orbits in $G/P$. The interpretation of the
conformal and projective compactifications of hyperbolic space from
this point of view and some generalizations are discussed in Example
\ref{ex2.2}. Next, we recall Wolf's theorem and describe several
examples of parabolic compactifications arising from it in Example
\ref{ex2.3}. Proposition \ref{prop2.4} describes the infinitesimal
structure of a neighborhood of an $H$--orbit in $G/P$ and Proposition
\ref{prop2.5} shows that the subgroup $H\subset G$ can always be
characterized as a stabilizer of an element in an appropriate
representation of $G$, thus leading to a parallel section of the
corresponding tractor bundle. The background on tractor bundles and
the machinery of BGG sequences we need is collected in Sections
\ref{2.6} and \ref{2.7}. Theorem \ref{thm2.6} describes the basic
relation between parallel sections of tractor bundles and solutions of
first BGG operators, while Proposition \ref{prop2.7} shows how the BGG
machinery can be used to recover information on the jets of such
solutions. The last part of Section \ref{2} introduces a
generalization of defining functions and defining densities for
hypersurfaces to defining sections of vector bundles for submanifolds
of higher codimension. These smooth objects provide a geometric and
analytic bridge between the different components, and in particular
they provide a very convenient way to formulate and prove many of our
results.

In the remaining two Sections of the article, we apply these general tools to the
study of two families of substantial examples of parabolic compactifications. Section
\ref{3} deals with the case that $H=SO(p,q)\subset SL(p+q,\Bbb R)=G$, which obviously
is covered by Wolf's theorem. We focus on the case that $P\subset G$ is a maximal
parabolic, so that $G/P$ is the Grassmannian $Gr(i,\Bbb R^{p+q})$ in which the
$H$--orbits are determined by rank and signature, and give some indication on how to
deal with more general flag varieties. The orbit structure and the infinitesimal
structure around an orbit on the Grassmannian is described in Proposition
\ref{prop3.1}. In particular, for the $i=p$, the $H$--orbit of positive subspaces is
the Riemannian symmetric space $SO(p,q)/S(O(p)\x O(q))$ which we take as the model
example for this section. This symmetric space clearly carries an $H$--invariant
Grassmannian structure (i.e.~a decomposition of its tangent bundle as a tensor
product) compatible with its Riemannian metric, and the main feature of the
compactification we construct is that this Grassmannian structure admits a smooth
extension to the boundary. The inner product stabilized by $H$ can be directly
converted into a parallel section of a tractor bundle, which is the main tool used to
analyze the compactification.

As a first application of the theory of parabolic geometries, we show in Theorem
\ref{thm3.2}, how the closures of $H$--orbits in the Grassmannian can be described as
zero--loci of first BGG solutions that admit a smooth extension to the boundary. The
main result in this section is a slice theorem, see Theorem \ref{thm3.3} and
Corollary \ref{cor3.3} which shows that locally, a neighborhood of an orbit can
always be described in terms of a product of the orbit with a neighborhood of zero in
the closure of the set of non--degenerate symmetric matrices of appropriate size and
signature in the space of all symmetric matrices. The key ingredient to this is the
construction of a defining section for the orbit obtained by restricting an
appropriate first BGG solution to an appropriate subbundle of the tautological bundle
on the Grassmannian. As a last result in this case, we show in Proposition
\ref{prop3.4} that for the boundary components of codimension one, one may avoid such
choices and find a canonical defining density, which is a solution of a first BGG
operator of order three.

In Section \ref{4} of the article, we similarly study the case that $H:=SO(n,\Bbb
C)\subset SO_0(n,n)=:G$, which again is covered by Wolf's theorem, and in particular
leads to a compactification of the Riemannian symmetric space $SO(n,\Bbb
C)/SO(n)$. The generalized flag varieties of $G$ are given by manifolds of isotropic
flags, and we quickly specialize to the case of the Grassmannians of maximal
isotropic subspaces, where the $H$--orbits are again described in terms of rank and
signature, but ranks always drop in steps of two in this case. Moreover, it is well
known that there are two such Grassmannians, given by self--dual respectively
anti--self--dual maximally isotropic subspaces. We briefly indicate how to deal with
more general isotropic Grassmannians and flag manifolds, for which the orbit
structure becomes significantly more complicated.

While initially there is no Hermitian structure in the setup, it turns out that
Hermitian matrices play a key role in the slice theorem for the $H$--orbits in this
case, see Theorem \ref{thm4.3}. While we also obtain a description of orbit closures
in terms of first BGG solutions in this case (see Theorem \ref{thm4.2}), we do not
see a way to construct a natural defining density for the hypersurface components of
the boundary.

\section{Parabolic compactifications}\label{2} 
In this section, we introduce the concept of a parabolic
compactification and explain the advantages of such
compactifications. Finally, we describe several sources for large
families of examples.

\subsection{Homogeneous compactifications}\label{2.1} 
A natural idea for constructing a compactification of some non--compact
space $X$ is to embed it into a compact space $Y$ and then form the
closure $\overline{X}$ in $Y$. This is of course compact and contains
$X$ as a dense subspace, thus defining a compactification of $X$ in
the usual topological sense. In case that $X$ is a smooth manifold one may
of course try to embed $X$ into a compact smooth manifold $Y$ and then
form the closure in there. However, at this level of generality, it is
very hard to control the ``boundary'' $\overline{X}\setminus X$ that
is added to $X$ in order to obtain the compactification. For example it may be
unclear whether this boundary inherits some kind of intrinsic smooth structure,
and $\overline{X}$ may be very badly behaved.

Let us specialize to the case that $X$ is a homogeneous space $X=H/K$,
where $H$ is a Lie group and $K\subset H$ is a closed subgroup. Then
one can try to exploit the homogeneous structure by embedding $H/K$
into a compact homogeneous space in an equivariant way. This leads to
the concept of a homogeneous compactification. 

\begin{definition}\label{def2.1}
Let $H$ be a Lie group and $K\subset H$ a closed subgroup. Then a
\textit{homogeneous compactification} of $H/K$ is defined by an
embedding $i:H\hookrightarrow G$, of $H$ as a closed subgroup of a Lie
group $G$, and   a closed subgroup $P\subset G$ such that $G/P$ is
compact and $P\cap H=K$.
\end{definition}

An embedding as in this definition descends to an embedding
$\underline{i}:H/K\hookrightarrow G/P$ of the homogeneous space $H/K$
into the compact homogeneous space $G/P$, which by construction is
$H$--equivariant. Hence one obtains a compactification of $H/K$ by
forming the closure $\overline{H/K}\subset G/P$, and we will also
refer to this closure as the homogeneous compactification of $H/K$.
%% (We allow the possibility that $H/K$ is equal to $G/P$, which of
%% course as a compactification is a trivial case.)

In this situation, we can already get some basic information on the
structure of the boundary $\overline{H/K}\setminus H/K$ that is
added to $H/K$ in order to obtain the compactification. Indeed, the
closed subgroup $H\subset G$ naturally acts on the homogeneous space
$G/P$ by the restriction of the canonical action of $G$. Using this,
we can prove the following result.

\begin{prop}\label{prop2.1}
  Consider a homogeneous compactification of $H/K$ given by the
  embedding $H/K\hookrightarrow G/P$. Then the boundary
  $\overline{H/K}\setminus H/K$ naturally is a union of $H$--orbits in
  $G/P$, each of which is an initial submanifold of $G/P$.
\end{prop}
\begin{proof}
  As a subset of $G/P$, the space $H/K$ of course is $H$--invariant,
  which readily implies that the closure $\overline{H/K}\subset G/P$
  is $H$--invariant, too. But given a smooth action of $H$ on a
  manifold, any invariant subset is  a union of orbits. On the other
  hand, since the orbits of a smooth action can be realized as leaves
  of a foliation (of non--constant rank), they are automatically
  initial submanifolds, compare with Theorem 5.14 in \cite{KMS}.
\end{proof}

This statement means that the smooth structure of the individual
$H$--orbits in $G/P$ is completely understood: Given one of the
orbits, say $\Cal O:=H\cdot gP\subset G/P$ let $L$ be the stabilizer
of $gP$ in $H$. Then there is an injective immersion $j:H/L\to G/P$
whose image coincides with $\Cal O$. So we can form $j^{-1}:\Cal O\to
H/L$ and for any manifold $M$, a function $f:M\to G/P$ with values in
$\Cal O$ is smooth if and only if $j^{-1}\o f:M\to H/L$ is
smooth.

This result is just a small first step. At this point the question of how the
different orbits are ``pieced together'', and what additional
structure each might have, remains to be resolved. 

\subsection{Parabolic Compactifications}\label{2.2} 

We now specialize homogeneous compactifications to the case that $G$ is
a semisimple Lie group and $P\subset G$ is a parabolic subgroup. It is
well known that this automatically implies that $G/P$ is
compact. Hence we consider an inclusion $H\hookrightarrow
G$, and put $K=H\cap P$. This makes the $H$--orbit of $o:=eP\in G/P$
isomorphic to $H/K$ and we consider the compactification of $H/K$
given by the closure of this orbit. In fact, we specialize things
a bit further, as follows.

\begin{definition}\label{def2.2}
Let $H$ be a Lie group and $K\subset H$ a closed subgroup. Then a
\textit{parabolic contactification} of the homogeneous space $H/K$ is a
homogeneous compactification, for which the group $G$ is semisimple,
the subgroup $P\subset G$ is parabolic and which has the property that
$H/K$ is open in $G/P$.
\end{definition}

The motivation for this specialization is the following. For a
parabolic subgroup $P$ in a semisimple Lie group $G$, the homogeneous
space $G/P$ carries a natural geometric structure, which is fairly
well understood. Indeed, $G/P$ is the homogeneous model of parabolic
geometries of type $(G,P)$. This means that immediately a large number
of geometric tools are available. Since parabolic geometries can be
restricted to open subsets, the homogeneous space $H/K$ inherits a
locally flat parabolic geometry of type $(G,P)$. Since this is an
$H$--invariant geometric structure on $H/K$, its existence is usually
clear in advance. However it is just one of the $H$--invariant
geometric structures available on $H/K$ and for this specific
geometric structure we get the crucial additional information that it
admits a smooth extension across the boundary of the
compactification. The same holds for all bundles and natural
operations associated to this structure. All together these provide powerful tools
to study the structure of the boundary and its relation to the geometry  $H/K$.

The inclusion $H\hookrightarrow G$  has a nice interpretation
in the language of parabolic geometries of type $(G,P)$. Indeed, it
defines a \textit{holonomy reduction} of the Cartan geometry $G\to
G/P$ in the sense studied in \cite{hol-red}. The main topic of
\cite{hol-red} is extending properties of a reduction of the
homogeneous model to cases of curved geometries. Hence any parabolic
contactification automatically comes with a notion of \textit{curved
  analog} which we describe next. 

Suppose that $M$ is a compact smooth manifold endowed with a parabolic
geometry of type $(G,P)$ and a holonomy reduction corresponding to
$H\hookrightarrow G$. Then one of the basic results of \cite{hol-red}
is that the holonomy reduction gives rise to a decomposition of $M$
into \textit{curved orbits} according to the decomposition of $G/P$
into $H$--orbits. In particular, the base--point $o=eP$ for which
$K=H\cap P$ defines a type of curved orbit. Using this, we can give
the definition of curved analogs.

\begin{definition}\label{def-curved}
  Consider a parabolic compactification defined by
  $H/K\hookrightarrow G/P$, so $K=H\cap P$. Then a curved analog of
  this compactification is defined as follows. Consider a parabolic geometry of type
  $(G,P)$ on a smooth manifold $M$ together with a holonomy reduction
  to the group $H$ such that the curved orbit of the type determined
  by $eP\in G/P$ is non--empty and has compact closure in $M$. Then this closure provides the compactification
  of the given curved orbit.
\end{definition}

The conditions on compactness of the closure of the curved orbit is of
course satisfied automatically if the ambient manifold $M$ is
compact. Since the developments in \cite{hol-red} are phrased in a
geometric language, they provide the basis for the geometric study of
parabolic compactification that we are initiating in this article.

\begin{example}\label{ex2.2}
  We start with two examples of parabolic compactifications for which
  curved analogs are already intensively studied in the literature
  (and for these there is even a more general notion of curved analog
  than that mentioned above). Then we provide one more general family
  of examples.

(1) Consider the connected group $G:=SO_0(n+2,1)$  defined by a non--degenerate
bilinear form $\langle\ ,\ \rangle$ of signature $(n+2,1)$ on $\Bbb
R^{n+3}$. Choose a vector $v_0\in\Bbb R^{n+3}$ such that $\langle
v_0,v_0\rangle=1$ and let $H\subset G$ be the stabilizer of $v_0$ in
$G$. Via the action on the orthocomplement $(v_0)^\perp$, the group
$H$ is identified with $SO_0(n+1,1)$.

It is well known that $G$ acts transitively on the space of future
directed isotropic rays in $\Bbb R^{n+3}$ which is diffeomorphic to
$S^{n+1}$ and that this provides the homogeneous model for oriented
Riemannian conformal structures in dimension $n+1$. Now given an
isotropic ray $\ell$, we can look at $\langle v,v_0\rangle$ for some
$v\in\ell$.  Whether this is positive, negative or zero is independent
of the choice of $v$ and it is easy to see that this strict sign
describes the full decomposition of the space of isotropic rays into
$H$--orbits. So there are two open orbits and one closed orbit. Taking
$P\subset G$ to be the stabilizer of a ray $\Bbb R_+\cdot v$ with
$\langle v,v_0\rangle>0$, we thus obtain a parabolic compactification
$H/(H\cap P)\hookrightarrow G/P$. Now $K:=H\cap P$ clearly coincides
with the stabilizer in $H$ of the ray obtained by projecting $\ell$
into $(v_0)^\perp$, and since $\ell$ is isotropic, this projection
must be spanned by a vector that is {\em timelike}, in that it is of
negative length according to $\langle\ ,\ \rangle$. Thus $K$ coincides
with the stabilizer of that timelike vector, and
$H/K\cong SO_0(n+1,1)/SO(n+1)$, so this is hyperbolic space of
dimension $n+1$.

On the other hand, the closed curved obit is just the space of
isotropic rays in $(v_0)^\perp$, so this is the sphere $S^n$ viewed as
a homogeneous space of $SO_0(n+1,1)$. In fact the parabolic
compactification obtained in this case is the conformal
compactification of hyperbolic space with the conformal sphere as its
boundary at infinity (i.e.~the Poincar\'e model of hyperbolic
space) \cite{Gover:P-E}. Curved analogs of this parabolic contactification are defined
as reductions of conformal holonomy of conformal manifolds of
dimension $n+1$. As discussed in Section 3.5 of \cite{hol-red}, the
results of \cite{Gover:P-E} show that curved holonomy reductions of
type $H\hookrightarrow G$ are equivalent to Poincar\'e--Einstein
metrics. The general concept of conformally compact metrics, as
introduced by R.~Penrose, can then be viewed as a further weakening of
the concept of this type of holonomy reduction.

\medskip

(2) Consider $G:=SL(n+2,\Bbb R)$ and $H:=SO(n+1,1)\subset G$ defined
by the choice of a Lorentzian metric $\langle\ ,\ \rangle$ on $\Bbb
R^{n+2}$. Then $G$ acts on the space of rays in $\Bbb R^{n+2}$, thus
providing the homogeneous model $S^{n+1}$ of oriented projective
structures in dimension $n+1$. As in example (1) above, the
$H$--orbits on $S^{n+1}$ are determined by the signature of the
restriction of $\langle\ ,\ \rangle$ to a ray. So again there are two
open orbits corresponding to positive or negative restriction and
there is one closed orbit corresponding to lines which are isotropic
for $\langle\ ,\ \rangle$.

Choosing $P$ to be the stabilizer of a timelike ray, we obtain
$K:=P\cap H\cong SO(n+1)$ so we obtain another parabolic
compactification of hyperbolic space $H/K\hookrightarrow G/P$. As in
example (1) above, the boundary is the space of isotropic lines, so
this is again the conformal sphere $S^n$. However, this time we obtain
the Klein model, which is a projective compactification of hyperbolic
space. As discussed in detail in \cite{ageom}, while the
compactifications from examples (1) and (2) of course are isomorphic
as topological compactifications, they are different from a geometric
point of view. Curved analogs of this parabolic compactification are
Klein--Einstein structures, see \cite{ageom} and Sections 3.1--3.3 of
\cite{hol-red}. This concept can be further generalized to projective
compactness of order $2$, as introduced in \cite{Proj-comp}. 

\medskip

(3) Generalizing the situation from (2), consider the Grassmannian
$Gr(p,\Bbb R^{p+q})$, which can be realized as $G/P$, with
$G=SL(p+q,\Bbb R)$ and $P$ the stabilizer of a $p$--dimensional
subspace $V$ in the standard representation $\Bbb R^{p+q}$ of
$G$. Consider the action of the subgroup $H:=SO(p,q)\subset G$ on the
Grassmannian. Choosing $V$ in such a way that the bilinear form
defining $H$ is positive definite on $V$, we see that $K:=H\cap P\cong
S(O(p)\times O(q))$, so $H/K$ is a Riemannian symmetric space. Linear
algebra implies that the $H$--orbit of $V$ in $G/P$ consists of all
subspaces for which the restricted bilinear form is positive definite,
so this orbit is clearly open, and we have constructed a parabolic
compactification of $H/K$.

The space $G/P$ is the homogeneous model for (almost) Grassmannian
structures of type $(p,q)$. This geometry is available on manifolds of
dimension $pq$ and on a manifold $M$ such a structure is basically
given by an identification of $TM$ with a tensor product of two
auxiliary vector bundles of rank $p$ and $q$, respectively. The
additional information completing this is a suitable identification of
the top exterior powers of the auxiliary bundles. Now in the
symmetric decomposition $\frak h=\frak k\oplus\frak m$, it is easy to
see that $\frak m$ can be identified with the space matrices of size
$p\x q$ endowed with the natural representation of
$\frak k\cong\frak o(p)\x\frak o(q)$. This shows that $H/K$ carries an
$H$--invariant almost Grassmannian structure, which is also nicely
compatible with the $H$--invariant Riemannian metric. Now this
Riemannian metric is complete, so there is no hope to smoothly extend
is across the boundary of $H/K$ in $G/P$. In contrast, the
Grassmannian structure on $H/K$ does extend across the boundary.
\end{example}

\subsection{Wolf's theorem}\label{2.3} 
We next describe a general scheme which can be used to generate large
families of parabolic compactifications. The basis for this is the
following theorem of J.~Wolf from 1974, which is the main result of
\cite{Wolf:finite}.
\begin{thm}\label{thm2.3}
  Let $G$ be a real semisimple Lie group, let $\th:G\to G$ be an
  involutive automorphism, and let $H\subset G$ be the fixed point
  group of $\th$. Then for any parabolic subgroup $P\subset G$, the
  $H$--action on the generalized flag manifold $G/P$ has only finitely
  many orbits. In particular, there are both open and closed
  $H$--orbits in $G/P$.
\end{thm}

The relevance of this result for our purposes is 
evident. Suppose that $H=G^\th$ is the fixed point group of an
involutive automorphism of a real semisimple Lie group. Then for any
parabolic subgroup $P\subset G$, there are open $H$--orbits in
$G/P$. Replacing $P$ by a conjugate subgroup if necessary, we may
assume that the $H$--orbit of $eP$ is open. Then using Proposition
\ref{prop2.1}, we get:

\begin{cor}\label{cor2.3}
Let $G$ be a semisimple Lie group, $\th:G\to G$ an involutive
automorphism and $H:=G^\th$ its fixed point group. Let $P\subset G$ be
a parabolic subgroup such that the $H$--orbit of $eP$ is open in
$G/P$. Then putting $K:=H\cap P$, we obtain a parabolic
compactification $H/K\hookrightarrow G/P$ whose boundary
$\overline{H/K}\setminus H/K$ is the union of finitely many
$H$--orbits in $G/P$, each of which is an initial submanifold.
\end{cor}

We next discuss two sources of examples of involutive automorphisms of
a classical real semisimple group $G$, for which we also get a
description of $H$ as a stabilizer.

\begin{example}\label{ex2.3}
  (1) Let $\Bbb K$ be $\Bbb R$, $\Bbb C$, or $\Bbb H$, with
  conjugation being defined as the identity on $\Bbb R$ or as the usual
  conjugations on $\Bbb C$ and $\Bbb H$. Suppose that
  $G=SL_{\Bbb K}(V)$ for a finite dimensional $\Bbb K$--vector space
  $V$. Then let $\langle\ ,\ \rangle$ be a non--degenerate sesquilinear
  form on $V$ and for $A\in GL(V)$ let $A^*$ be the adjoint of $A$
  with respect to $\langle\ ,\ \rangle$, so
  $\langle Av,w\rangle=\langle v,A^*w\rangle$ for all $v,w\in V$. This
  definition readily implies that $(AB)^*=B^*A^*$ and hence
  $\th(A):=(A^{-1})^*$ defines an involutive automorphism of $G$ whose
  fixed point group $H$ is the special unitary group of
  $\langle\ ,\ \rangle$. Depending on $\Bbb K$, this group is
  (isomorphic to) $SO(p,q)$, $SU(p,q)$ or $Sp(p,q)$, where $(p,q)$ is
  the signature of $\langle\ ,\ \rangle$.

  Now the generalized flag varieties $G/P$ in this case are just the
  manifolds of all partial flags
  $V_1\subset V_2\dots\subset V_k\subset V$ of $\Bbb K$--subspaces of
  fixed dimensions. In the simplest case, $G/P$ is the Grassmannian of
  $i$--dimensional subspaces $V_1\subset V$. Here linear algebra shows
  that the $H$--orbits on $G/P$ are characterized by the rank and
  signature of the restriction of $\langle\ ,\ \rangle$ to $V_1$ and
  the open orbits are exactly those for which the restriction is
  non--degenerate. Taking the signatures to be $(p,q)$ on $V$ and
  $(r,i-r)$ on $V_1$ (which implies that $\langle\ ,\ \rangle$ is
  non--degenerate on $V_1$), we get that $V=V_1\oplus V_1^\perp$. This
  decomposition is preserved by the stabilizer of $V_1$ in $H$, which,
  for $\Bbb K=\Bbb R$ is thus easily seen to be isomorphic to
  $S(O(r,i-r)\x O(p-r,q-i+r))$ and similarly in the other cases. In
  particular taking $i=r=p$ and $\Bbb K=\Bbb C$, we obtain a parabolic
  compactification of the Hermitian symmetric space
  $SU(p,q)/S(U(p)\x U(q))$, which is easily seen to coincide with the
  Bailey--Borel compactification. For $\Bbb K=\Bbb R$ and $\Bbb H$, we
  obtain analogous compactifications for Riemannian symmetric spaces
  of $SO(p,q)$ and $Sp(p,q)$.
  
  The first parts of this directly generalize to all flag
  manifolds. For a flag
  $V_1\subset V_2\subset\dots\subset V_k\subset V$, the rank and
  signature of the restriction of $\langle\ ,\ \rangle$ to each
  constituent $V_j$ of the flag is of course constant along each
  orbit. Moreover, linear algebra easily shows that the flags for
  which each of these restrictions is non--degenerate of fixed
  signature form a single open orbit and that these are the only open
  orbits. In such a case, putting $\dim(V_j)=i_j$ for each $j$ and
  denoting the signature of the restriction by $(r_j,s_j)$, we thus
  obtain, for $\Bbb K=\Bbb R$, parabolic compactifications of the
  homogeneous spaces
$$
SO(p,q)/S(\textstyle\prod_{j=1}^{k+1}O(r_j-r_{j-1},s_j-s_{j-1})),
$$
where we put $r_0=s_0=0$, $r_{k+1}=p$ and $s_{k+1}=q$. This works in
the same way for the other ground fields. Notice that these are not
symmetric spaces for $k\geq 2$. It turns out that, in contrast to the
case of Grassmannians, the degenerate orbits are not determined by the
ranks and signatures of the restrictions of $\langle\ ,\ \rangle$ to
the constituents of the flag any more, so things get more
complicated. We will discuss this further in Section \ref{3.1} below.

\medskip

(2) The second source of involutions is even simpler and more
versatile. Again we put $\Bbb K=\Bbb R$, $\Bbb C$ or $\Bbb H$ and we
fix a $\Bbb K$--vector space $V$. For $G$ we either take the group
$SL_{\Bbb K}(V)$ or the subgroup of $SL_{\Bbb K}(V)$ preserving a
bilinear form $b$ on $V$ (which may be either symmetric or skew
symmetric, and either $\Bbb K$--bilinear or Hermitian). Then suppose
that $J:V\to V$ is either $\Bbb K$--linear or conjugate linear such
that $J^2=\ep \id_V$, where $\ep=\pm 1$. Then the map
$\th(A):=\ep JAJ$ defines an involutive automorphism on
$SL_{\Bbb K}(V)$. If a bilinear form $b$ is involved, then depending
on the properties of $b$, one has to require either that
$b(Jv,Jw)=b(v,w)$ or that $b(Jv,Jw)=-b(v,w)$ to ensure that $\th$
restricts to an involutive automorphism of the subgroup $G$. In any
case, the fixed point group $H$ of $\th$ then consists of those
elements $A\in G$ which commute with $J$.

The simplest instance of this is provided by identifying $\Bbb C^n$
with $\Bbb R^{2n}$ and viewing multiplication by $i$ as a linear
endomorphism $J$ of $\Bbb R^{2n}$ such that $J^2=-\id$. Starting from
$G=SL(2n,\Bbb R)$, we easily conclude from Wolf's theorem that
$H:=SL(n,\Bbb C)\subset G$ acts with finitely many orbits on each
partial flag manifold. In particular, we can look at the Grassmannian
$Gr(n,\Bbb R^{2n})$ for which linear algebra implies that there is
only one open $H$--orbit, namely the one consisting of those subspaces
$W$ which are totally real in the sense that $W\cap J(W)=\{0\}$.
Taking $P\subset G$ to be the stabilizer of $\Bbb R^n\subset\Bbb C^n$,
we see that $H\cap P=SL(n,\Bbb R)\subset SL(n,\Bbb C)$. Thus this
example gives rise to a parabolic compactification of the symmetric
space $SL(n,\Bbb C)/SL(n,\Bbb R)$.

Similarly, we can identify $\Bbb H^n$ with $\Bbb C^{2n}$, endowed with
the complex structure provided by $i\in\Bbb H$, and view
multiplication by $j\in\Bbb H$ as a conjugate linear map $J$ on $\Bbb
C^{2n}$ such that $J^2=-\id$. Hence we conclude that the subgroup
$SL(n,\Bbb H)$ acts with finitely many orbits on each generalized flag
manifold of $G:=SL(2n,\Bbb C)$. 

Bringing bilinear forms into the game, we obtain further interesting
examples. In the situation of $\Bbb C^n\cong\Bbb R^{2n}$, we can
consider a complex bilinear form $b$ on $\Bbb C^n$ and view its
imaginary part as a real symmetric bilinear form on $\Bbb R^{2n}$. It
is easy to see that this bilinear form has split signature $(n,n)$ and
clearly the involution defined by $J$ can be restricted to the
orthogonal group $G:=SO_0(n,n)$ of this split signature form. A
complex linear map which preserves the imaginary part of $b$ preserves
the whole form $b$, which easily implies that as a fixed point group,
we obtain $H:=SO(n,\Bbb C)\subset G$. We will discuss this case in
more detail in Section \ref{4}, where we shall in particular see that
it gives rise to a parabolic compactification of
$SO(n,\Bbb C)/SO(n)$, which is a Riemannian symmetric space of
non--compact type.

As a final example, let $\langle\ ,\ \rangle$ be the standard quaternionic Hermitian
form on $\Bbb H^n$ and consider the form $\om$ on $\Bbb H^{2n}$ defined by
$\om(\binom{p_1}{p_2},\binom{q_1}{q_2}):=\langle p_1,q_2\rangle-\langle
p_2,q_1\rangle$.
One immediately verifies that this form is skew--Hermitian in the quaternionic sense,
so the group of quaternionically linear automorphisms of $\Bbb H^{2n}$ that preserve
$\om$ form a Lie group $G$ that is commonly denoted by $SO^*(4n)$ and is a real form
of $SO(4n,\Bbb C)$. On the other hand $J\binom{p_1}{p_2}=\binom{p_2}{p_1}$ defines an
$\Bbb H$--linear automorphism of $\Bbb H^{2n}$ such that $\om(Jp,Jq)=-\om(p,q)$ holds
for all $p,q\in\Bbb H^{2n}$. Now one easily verifies that the Lie algebra $\frak g$
of $G$ consists of block matrices of the form
$\begin{pmatrix} A & B \\ C & -A^*\end{pmatrix}$, where $A$, $B$, and $C$ are
quaternionic $n\x n$--matrices, $B^*=B$ and $C^*=C$. As we have seen above, the fixed
point group $H$ of the involution determined by $J$ consists of all matrices
commuting with $J$. On the Lie algebra level, this means that for the above block
form, we in addition get $A=-A^*$ and $C=B$, which easily implies that $H$ is
identified with $GL(n,\Bbb H)$ via its action on the subspace of all elements of the
form $\binom{p_1}{p_1}$, which form an $n$--dimensional quaternionic subspace of
$\Bbb H^{2n}$.

The generalized flag varieties of $G$ are the manifolds of flags of quaternionic
subspaces in $\Bbb H^{2n}$, which are isotropic for $\om$. Looking at the example of
the Grassmannian of maximal (i.e.~$n$--dimensional) isotropic subspaces, we see that
the dimension of the span of a subspace and its image under $J$ is constant on
$H$--orbits. In particular, if $V$ is a maximally isotropic subspace such that $V$
and $J(V)$ span all of $\Bbb H^{2n}$ the $H$--orbit of $V$ is open. Now consider the
subspace of all vectors of the form $\binom{p_1}{0}$, which clearly has that
property. The Lie algebra of the stabilizer of that subspace in $H$ are matrices of
the form $\begin{pmatrix} A & 0\\ 0 & A \end{pmatrix}$ with $A^*=-A$. This shows that
the stabilizer will be a subgroup isomorphic to $Sp(n)\subset GL(n,\Bbb H)$. Thus
this example leads to a parabolic compactification of the Riemannian symmetric space
$GL(n,\Bbb H)/Sp(n)$.
\end{example}

\subsection{Orbits and infinitesimal transversals}\label{2.4}
Given a parabolic compactification of $H/K$ coming from an embedding
into $G/P$, we next discuss some general tools that can be used to
identify the $H$--orbits in $G/P$ as well as an infinitesimal model
for the local structure around such an orbit. The first thing to point
out here is that there is an issue of ``relative position'' between
$H$ and $P$ that may look unfamiliar. Looking at
$H=SO(p,q)\subset SL(p+q,\Bbb R)=G$ one usually does not specify the
bilinear form preserved by $H$ explicitly, since any two choice lead
to conjugate subgroups. Likewise, looking at the stabilizer
$P\subset G$ of an $i$--dimensional subspace of $\Bbb R^{p+q}$, the
subspace is not specified explicitly for the same reason. However,
fixing both data at the same time, the relative position is encoded in
the rank and signature of the restriction of the chosen bilinear form
to the chosen subspace, which is clearly invariant under simultaneous
conjugations of both subgroups.

To deal with these issues in practice, we will fix an $H$--orbit $\Cal
O\subset G/P$. Then we will choose a subgroup $H_{\Cal O}\subset G$
conjugate to $H$ in such a way that the orbit of the base point
$o:=eP\in G/P$ is isomorphic to $\Cal O$. Basically, this means that
we keep $P\subset G$ fixed and arrange $H_{\Cal O}$ in such a way that
the ``right'' relative position is achieved. Having done that, that
stabilizer of $o$ in $H_{\Cal O}$ coincides with $H_{\Cal O}\cap P$
and thus $\Cal O\cong H_{\Cal O}/(H_{\Cal O}\cap P)$. In particular,
having determined $H_{\Cal O}\cap P$, we can readily read off the
codimension of $\Cal O$ in $G/P$.  

Identifying $\Cal O$ with a homogeneous space of $H_{\Cal O}$ makes
all the standard tools for analyzing the geometry of homogeneous
spaces available in our situation. In particular, we can easily obtain
an infinitesimal model for a neighborhood of $\Cal O$ in $G/P$. To
formulate the result, observe that $H_{\Cal O}\subset G$ acts on
$\frak g$ by the restriction of the adjoint action. The subgroup
$H_{\Cal O}\cap P$ leaves the subspace $\frak p\subset\frak g$
invariant and hence naturally acts on $\frak g/\frak p$. Likewise,
$H_{\Cal O}\cap P$ naturally acts on $\frak h_{\Cal O}/(\frak h_{\Cal
  O}\cap\frak p)$. Now the inclusion $\frak h_{\Cal
  O}\hookrightarrow\frak g$ descends to an injection $\frak h_{\Cal
  O}/(\frak h_{\Cal O}\cap\frak p)\hookrightarrow \frak g/\frak p$
which is equivariant for the actions of $H_{\Cal O}\cap P$, so again
there is a natural representation on the quotient. 

\begin{prop}\label{prop2.4}
  Let $\Cal O\cong H_{\Cal O}/(H_{\Cal O}\cap P)\subset G/P$ be an
  $H$--orbit. Then the quotient bundle
  $T(G/P)|_{\Cal O}/T\Cal O\to\Cal O$ is the homogeneous vector bundle
  induced by the representation of $H_{\Cal O}\cap P$ on
  $(\frak g/\frak p)/(\frak h_{\Cal O}/(\frak h_{\Cal O}\cap\frak p))$
  described above.
\end{prop}
\begin{proof}
  By assumption, $\Cal O$ is the $H_{\Cal O}$--orbit of $o=eP\in
  G/P$. Clearly, the derivatives of the actions of elements of
  $H_{\Cal O}$ on $G/P$ make $T(G/P)|_{\Cal O}$ into a homogeneous
  vector bundle. It is well known that $T_o(G/P)=\frak g/\frak p$ with
  the natural representation of $P$ coming from the adjoint
  representation.  For the subgroup $H_{\Cal O}\cap P\subset P$, we
  obtain the representation described above, so the general
  description of homogeneous bundles shows that $T(G/P)|_{\Cal O}\cong
  H_{\Cal O}\x_{H_{\Cal O}\cap P}(\frak g/\frak p)$. Thus the result
  follows from the standard description of $T\Cal O$ as a homogeneous
  vector bundle.
\end{proof}

\begin{remark}\label{rem2.4}
  If the action of $H_{\Cal O}$ on $G/P$ were proper, then this result
  would directly imply a description of a tubular neighborhood of
  $\Cal O$ in $G/P$ and its decomposition into $H_{\Cal O}$--orbits
  via the slice theorem for proper actions, see \cite{Palais}. By
  properness of the action one would get a positive definite inner
  product on $\frak g/\frak p$ which is invariant under the action of
  $H_{\Cal O}\cap P$, thus defining a $H_{\Cal O}$--invariant
  Riemannian metric on $G/P$. The normal bundle of the orbit then is
  induced by the orthocomplement of
  $\frak h_{\Cal O}/(\frak h_{\Cal O}\cap\frak p)$ in
  $\frak g/\frak p$, and a slice is obtained from exponentiating
  this. However, since $G/P$ is compact, properness of the action
  would imply that $H_{\Cal O}$ has to be compact, which is absurd in
  our setting. Hence the general slice theorem is never applicable in
  our situation.

Still we shall derive local descriptions of a neighborhood of $\Cal O$
in $G/P$ which are similar to the one obtained from the slice theorem
in several cases below. Thus in these cases, the orbit structure close
to $\Cal O$ again is described in terms of orbits of the isotropy
group $H_{\Cal O}\cap P$ on the representation $(\frak g/\frak
p)/(\frak h_{\Cal O}/(\frak h_{\Cal O}\cap\frak p))$.
\end{remark}

\subsection{Characterizing the subgroup $H$}\label{2.5}
As indicated above, we want to study a parabolic compactification of a
homogeneous space $H/K$ using the restriction of the natural parabolic
geometry on the ambient homogeneous space $G/P$. One of the important
features of parabolic geometries is the existence of a special type of
geometric objects, which are called \textit{tractors}. These are
closely related to representations of the group $G$, so $H$--invariant
elements in representations of $G$ will be of particular importance
for the further development. 

Observe first that in all the cases discussed in Example \ref{ex2.3},
we can characterize $H$ as the stabilizer of an appropriate element in
a relatively simple representation of $G$. Indeed, in the situation of
part (1) of Example \ref{ex2.3}, we put $G=SL_{\Bbb K}(V)$ and $H$ is
the special unitary group of a Hermitian form $\langle\ ,\ \rangle$ of
$G$. Hermitian bilinear forms form a subrepresentation of $S^2V^*$, so
we have characterized $H$ as the stabilizer in $G$ of an element in
that representation. Likewise, in the situation of part (2) of of
Example \ref{ex2.3}, we have $G\subset SL_{\Bbb K}(V)$ and an
endomorphism $J$ of $V$. This defines an element in (an appropriate
subrepresentation of) the $G$--representation $V^*\otimes V$ of
endomorphisms of $V$, whose stabilizer in $G$ is $H$.

We next show that a similar characterization is available in a fairly
general situation. 

\begin{prop}\label{prop2.5}
Let $G$ be a classical simple Lie group with standard representation
$W$ and let $H\subset G$ be a connected closed semisimple Lie
subgroup. Suppose that either $W$ is a complex representation of $G$
which is irreducible for $H$ or that the complexification of $W$ is
irreducible for $H$.

Then denoting by $\frak g$ the Lie algebra of $G$, the representation
$V:=\frak{sl}(\frak g)$ of $G$ contains an element $v_0$, whose
stabilizer in $G$ is $H$.
\end{prop}
\begin{proof}
Via the restriction of the adjoint representation of $G$, the Lie
algebra $\frak g$ is a representation of $H$, and of course $\frak
h\subset\frak g$ is an $H$--invariant subspace. Since $\frak h$ is
semisimple, there is an $\frak h$--invariant subspace $E\subset\frak
g$ which is complementary to $\frak h$. Since $H$ is connected, the
subspace $E$ is $H$--invariant. Now let $v_0\in L(\frak g,\frak g)$ be
the map which acts as the identity on $\frak h$ and by multiplication
by $-\tfrac{\dim(\frak h)}{\dim(W)}$ times the identity on $W$, where
the factor is chosen to ensure that $v_0$ is trace--free.

The natural action of $H$ on $V:=\frak{sl}(\frak g)$ is given by
$h\cdot v=\Ad(h)\o v\o \Ad(h)^{-1}$, which readily implies that $v_0$
is $H$--invariant. Since $H$ is connected, we can show that it
coincides with the stabilizer of $v_0$ in $G$ by proving that if
$A\in\frak g$ satisfies $0=A\cdot v_0$, then $A\in\frak h$. Now of
course $A\cdot v_0=\ad(A)\o v_0-v_0\o\ad(A)$, so $A\cdot v_0=0$ means
that $\ad(A)$ commutes with $v_0$. In particular, $\ad(A)$ must
preserve the eigenspaces of $v_0$, so $\ad(A)(\frak h)\subset\frak
h$. Now the restriction of $\ad(A)$ to $\frak h$ is a derivation by
the Jacobi--identity, and since $\frak h$ is semisimple, any such
derivation is inner. Hence there is an element $B\in\frak h$ such that
$\ad(A)|_{\frak h}$ coincides with $\ad(B)|_{\frak h}$. Hence
$A-B\in\frak g$ has the property that $[A-B,X]=0$ for any $X\in\frak
h$. This means that, as an endomorphism of $W$, $A-B$ commutes with
each element of $\frak h$. Passing to the complexification if
necessary, irreducibility implies that $A-B$ must be a multiple of the
identity. Since $\frak g$ is simple, our assumptions imply that it
consists of tracefree maps on $W$ respectively its complexification,
which implies $A=B$.
\end{proof}

\begin{remark}\label{rem2.5}
Evidently, some choice is involved in the construction used in the
proof of Proposition \ref{prop2.5}. One could actually use a finer
decomposition of the $H$--invariant complement $W$ to $\frak h$ in
$\frak g$, say into $\frak h$--isotypical components or into $\frak
h$--irreducibles, and then choose a map acting by different scalars on
the individual components. Any such choice works as long as $\frak
h\subset\frak g$ is one of the eigenspaces of $v_0$.

The disadvantage of the general construction in Proposition
\ref{prop2.5} is that the representation $\frak{sl}(\frak g)$ used
there is already fairly complicated. So while this characterization of
the subgroup $H$ is also available in the situations discussed in
Example \ref{ex2.3} it will be much more efficient to work with the
simpler characterizations available in these cases.  
\end{remark}

\subsection{Tractor bundles and the BGG machinery}\label{2.6} 
We next review the machinery for parabolic geometries we need, referring
to Chapter 3.2 of \cite{book} for details. Let $G$ be a semisimple Lie
group, $P\subset G$ a parabolic subgroup and let
$\frak p\subset\frak g$ be the Lie algebras. Then we can form the
\textit{reductive Levi--decomposition}
$\frak p=\frak g_0\oplus\frak p_+$ of $\frak p$ into a reductive Lie
subalgebra $\frak g_0\subset\frak p$ and a nilpotent ideal
$\frak p_+\subset\frak p$. Next, one defines a closed subgroup
$G_0\subset P$ as consisting of those elements whose adjoint actions
preserve this decomposition. It then turns out that the map
$(g_0,Z)\mapsto g_0\exp(Z)$ defines a diffeomorphism
$G_0\x\frak p_+\to P$, so in particular $P_+:=\exp(\frak p_+)$ is a
nilpotent normal subgroup in $P$ such that $P/P_+\cong G_0$.

This has consequences for the representation theory of $P$. The Lie
algebra $\frak p_+$ has to act trivially on any irreducible
representation of $P$. In particular, any completely reducible
representation of $P$ is obtained by trivially extending a completely
reducible representation of the reductive group $G_0$. (For $G_0$,
complete reducibility of a representation just means that the center
acts diagonalizably.) On the other hand, a general representation $V$
of $P$ inherits a $P$--invariant filtration of the form
$V=V^0\supset V^1\dots \supset V^N\supset\{0\}$, which is defined
recursively by $V^{i+1}=\frak p_+\cdot V^i\subset V^i$. Assuming that
the center of $G_0$ acts diagonalizably on $V$, each of the subsequent
quotients $V^i/V^{i+1}$ is a completely reducible representation of
$P$. In particular, $V/V^1$ is the canonical \textit{completely
  reducible quotient} of $V$.

Via forming associated bundles, representations of $P$ give rise to
natural vector bundles on manifolds endowed with a parabolic geometry
of type $(G,P)$. In the case of the homogeneous model, these are the
usual homogeneous vector bundles, so $V$ corresponds to the bundle
$\Cal V:=G\x_P V\to G/P$. An equivariant map between two
representations of $P$ induces a vector bundle map between the
corresponding bundles. In particular, the $P$--invariant filtration
$\{V^i\}$ of $V$ gives rise to a filtration of $\Cal V$ by smooth
subbundles $\Cal V^i\subset \Cal V$ such that each of the successive
quotients $\Cal V^i/\Cal V^{i+1}$ is a completely reducible bundle. Here we
say that a natural bundle is \textit{completely reducible} if it is
induced by a completely reducible representation of $P$.

Another important class of associated bundles are \textit{tractor
  bundles}, which correspond to representations of $P$ which are
obtained as restrictions of representations of $G$. An important
feature of a tractor bundle is that, on general parabolic geometries,
it carries a natural linear connection called the \textit{tractor
  connection}. On the homogeneous model, the situation is even easier,
see Section 1.5.7 of \cite{book}. Given a representation $V$ of $G$,
the corresponding bundle $\Cal V=G\x_P V$ admits a canonical
trivialization and the tractor connection is the flat connection
induced by this trivialization. In particular, any element $v\in V$
defines a section $s=s_v\in\Ga(\Cal V)$ which is parallel for the
tractor connection. In particular, we can apply this to $H$--invariant
elements as discussed in Section \ref{2.5}.

The machinery of BGG sequences, which was developed in \cite{CSS-BGG}
and \cite{Calderbank-Diemer}, relates parallel sections of a tractor
bundle $\Cal V$ to sections of the completely reducible quotient
$\Cal H_0:=\Cal V/\Cal V^1$ which satisfy a certain system of linear
partial differential equations. Again this works for general parabolic
geometries and assumes a particularly simple form on the homogeneous
model. We summarize what we need here: The projection $V\to V/V^1$
induces a natural bundle map $\Pi:\Cal V\to\Cal H_0$, which in turn
induces a tensorial operator $\Ga(\Cal V)\to\Ga(\Cal H_0)$ on the
spaces of sections that will also be denoted by $\Pi$. The key fact
for the BGG machinery is that there is a natural differential operator
$S:\Ga(\Cal H_0)\to\Ga(\Cal V)$ which splits this tensorial
projection, i.e.~we get $\Pi(S(\si))=\si$ for each
$\si\in\Ga(\Cal H_0)$. Apart from this splitting property, there is
only one more property needed to characterize the operator $S$, namely
that $\nabla S(\si)$ always has values in a certain natural subbundle
$\ker(\partial^*)\subset T^*(G/P)\otimes\Cal V$.

Now there is a natural completely reducible quotient bundle $\Cal H_1$
of $\ker(\partial^*)\subset T^*(G/P)\otimes\Cal V$ and we denote by
$\pi_H$ both the corresponding bundle projection and the induced
tensorial operator on sections.  Putting these ingredients together,
we can define a differential operator
$D:\Ga(\Cal V/\Cal V_1)\to\Ga(\Cal H_1)$ by
$D(\si):=\pi_H(\nabla S(\si))$. This is the \textit{first
  BGG--operator} determined by $\Cal V$. A crucial feature is that the
representation inducing $\Cal H_1$ can be determined by purely
algebraic methods from $V$. It is given as a Lie algebra homology
group which can be computed using Kostant's theorem, but the details on
this are not relevant for our purposes. The main point is that knowing
$V$, the bundle $\Cal H_1$ as well as the order and the principal part
of the first BGG operator can be determined algorithmically. We will
not go into details on how this is done, but just state the
corresponding results when we need them. Let us collect the
information we will need for the further development.

\begin{thm}\label{thm2.6}
  Let $G$ be a semisimple Lie group, $P\subset G$ be a parabolic
  subgroup and $G/P$ the corresponding generalized flag manifold. Let
  $V$ be a representation of $G$, $\Cal V\to G/P$ the corresponding
  tractor bundle, $\Cal H_0\to G/P$ the canonical completely reducible
  quotient of $\Cal V$ and $\Pi:\Cal V\to\Cal H_0$ the corresponding
  projection.

Then there is a computable completely reducible natural bundle $\Cal
H_1$ and an invariant differential operator $D:\Ga(\Cal H_0)\to
\Ga(\Cal H_1)$ such that for $s\in\Ga(\Cal V)$ the following are
equivalent 
\begin{itemize}
\item[(a)] $s$ is parallel for the tractor connection $\nabla^\Cal V$
  on $\Cal V$.
\item[(b)] For $\si:=\Pi(s)$ we have $D(\si)=0$ and $s=S(\si)$. 
\item[(c)] There is an element $v\in V$ such that $s$ corresponds to
  the constant function $v$ in the natural trivialization of $\Cal
  V$. 
\end{itemize}
\end{thm}

In the situation of a parabolic compactification coming from an
embedding $H\hookrightarrow G$, we can apply this theorem to an
$H$--invariant element in some representation $V$ of $G$. This element
gives rise to a parallel section of the tractor bundle $\Cal V$
corresponding to $V$, which in turn projects to a solution of a first
BGG operator. As we shall see below such solutions can be used to
separate $H$--orbits in $G/P$ and thus understand the boundary
obtained by the compactification. Moreover, we can use such sections
to construct local coordinates on $G/P$ which are nicely adapted to
the decomposition into $H$--orbits, thus describing the local
structure of the $H$--space $G/P$. For these applications, we
crucially exploit that the parallel tractor captures information about
the jet of the underlying BGG solution in a nice way.

It should be remarked that Theorem \ref{thm2.6} partly generalizes to
curved geometries. In particular, a parallel section of a tractor
bundle always projects to a solution of the first BGG operator,
although in the curved case not all solutions are obtained in that
way. In any case, the version of Theorem \ref{thm2.6} for curved
geometries is sufficient to deal with curved analogs of parabolic
compactifications in the sense of Definition \ref{def-curved}.

\subsection{Recovering jet information}\label{2.6a}
The relation between a parallel tractor and the jet of the underlying
solution of a first BGG operator can be described in great detail, but
quite a bit of background is needed to formulate these results. Thus
we restrict to a very special case, which is sufficient to deal with
the examples discussed below. In particular, we only discuss the case
that the parabolic subgroup $P\subset G$ corresponds to a so--called
$|1|$--grading $\frak g=\frak g_{-1}\oplus\frak g_0\oplus\frak g_1$,
which equivalently means that the nilradical $\frak p_+$ of its Lie
algebra is abelian. This allows us to avoid the use of weighted jets
and of filtrations of the tangent bundle, and to formulate a uniform
result without having to distinguish cases. Here we work in the
setting of curved geometries, since restricting to the homogeneous
model does not provide any simplification. Thus we consider a
Cartan geometry $(p:\Cal G\to M,\om)$ of type $(G,P)$ which satisfies
the usual conditions of regularity and normality (cf. \cite{hol-red,book}).

Let $V$ be a representation of $G$ and let $\{V^i\}$ be the
$P$--invariant filtration of $V$. Then the definitions easily imply
that $\frak g_j\cdot V^i\subset V^{i+j}$ for $j=-1,0,1$. Hence for
$v\in V^i$ we can map $X\in\frak g_{-1}$ to $X\cdot v\in V^{i-1}$ and
the class of this modulo $V^i$ depends only on the class of $v$ modulo
$V^{i+1}$. Thus for each $i$, we obtain a well-defined linear map
$\partial:V^i/V^{i+1}\to \frak{g}_{-1}^*\otimes V^{i-1}/V^i$, which is easily
seen to be $P$--equivariant. Since $\frak g_{-1}\cong\frak g/\frak p$
and $TM=\Cal G\x_P(\frak g/\frak p)$, the natural bundle corresponding
to $\frak g_{-1}^*$ is the cotangent bundle. Hence denoting by
$\Cal V=\Cal G\x_P V$ the tractor bundle induced by $V$ and by
$\Cal V^i\subset \Cal V$ the subbundle corresponding to $V^i$, we get
natural bundle maps
$\partial: \Cal V^i/\Cal V^{i+1}\to T^*M\otimes \Cal V^{i-1}/\Cal V^i$
for each $i$. By definition $\partial$ is the zero map for $i=0$ and
it is well known that it is injective for $i>0$.

To formulate the result, we need some facts on Weyl structures for
parabolic geometries as discussed in Chapter 5 of \cite{book}. Weyl
structures generalize the choice of a metric in a conformal class (or,
more generally of a Weyl connection). They form a conceptual way to
describe parabolic geometries as an equivalence class of simpler
additional structures. Such structures always exist and choosing one
of them, one gets an induced linear connection (the \textit{Weyl
  connection}) on any natural bundle. Moreover, one also gets an
isomorphism from any natural bundle to its associated graded bundle (with
respect to the natural $P$--invariant filtration of the inducing
representation), called a \textit{splitting of the filtration}.

Suppose now that $s$ is a section of the tractor bundle $\Cal V\to M$
induced by a representation $V$ of $G$. Let us denote by $\{V^i\}$ the
$P$--invariant filtration of $V$ and by $\Cal V^i\subset\Cal V$ the
subbundle corresponding to $V^i$. Via the natural projection
$\Pi:\Cal V\to\Cal V/\Cal V^1$, the section $s$ canonically determines
a section $\si\in\Ga(\Cal V/\Cal V^1)$. Now choosing a Weyl structure
for the geometry in question, we obtain a corresponding Weyl
connection on $\Cal V$ and each of the subquotient bundles
$\Cal V^i/\Cal V^{i+1}$. On the other hand, the splitting of the
filtration defines an isomorphism
$\Cal V\to \oplus_{i\geq 0}\Cal V^i/\Cal V^{i+1}$. Under this
isomorphism, $s$ corresponds to a family of sections of the
subquotient bundles. This isomorphism has the property that it maps
each $\Cal V^j$ to $\oplus_{i\geq j}\Cal V^i/\Cal V^{i+1}$ and the
component in $\Cal V^j/\Cal V^{j+1}$ is given by the natural quotient
map. In particular, the component in $\Ga(\Cal V/\Cal V^1)$ of the
section corresponding to $s$ coincides with $\si$. Moreover, if for a
point $x\in M$, we have $\si(x)=0$, then $s(x)\in \Cal V^1_x$, so the
value of the component in $\Ga(\Cal V^1/\Cal V^2)$ at the point $x$ is
$s(x)+\Cal V^2_x$, and does not depend on the choice of Weyl structure.

\begin{prop}\label{prop2.6}
  Suppose that $P\subset G$ corresponds to a $|1|$--grading of
  $\frak g$. Consider a representation $V$ of $G$ endowed with its
  natural $P$--invariant filtration $\{V^i\}$. Let $M$ be a manifold
  endowed with a parabolic geometry of type $(G,P)$, let $\Cal V\to M$
  be the tractor bundle determined by $V$, and let
  $\Cal V^i\subset\Cal V$ be the smooth subbundle corresponding to
  $V^i\subset V$. Consider the natural bundle map
  $\partial:\Cal V^1/\Cal V^2\to T^*M\otimes\Cal V^0/\Cal V^1$ as
  defined above.

  Let $s\in\Ga(\Cal V)$ be a section that is parallel for the
  canonical tractor connection, and put
  $\si:=\Pi(s)\in\Ga(\Cal V/\Cal V^1)$. Choose a Weyl structure with
  Weyl connection $\nabla$, and let $\mu\in\Ga(\Cal V^2/\Cal V^1)$ be
  the component of the image of $s$ under the splitting of the
  filtration as described above. Then for any point $x\in M$, we have
  $\nabla\si(x)=-\partial(\mu(x))$.
 
  In particular, if $\si(x)=0$ (which implies that both $\nabla\si(x)$
  and $\mu(x)$ are independent of the choice of Weyl structure), we
  conclude that $\si$ has vanishing one--jet in $x$ if and only if
  $s(x)\in\Cal V^2_x\subset\Cal V_x$.
\end{prop}
\begin{proof}
% The splitting of the filtration of $\Cal V$ has the property that its
% restriction to $\Cal V^i\subset\Cal V$ has values in $\oplus_{j\geq
%  i}\Cal V^j/\Cal V^{j+1}$ and the component in $\Cal V^i/\Cal
% V^{i+1}$ coincides with the canonical projection. So in partiular, the
% component of $s$ in $\Ga(\Cal V/\Cal V^1)$ under this identification
% is $\si$. 
  A description of the tractor connection $\nabla^{\Cal V}$ under the
  isomorphism $\Ga(\Cal V)\cong\oplus_i\Ga(\Cal V^i/\Cal V^{i+1})$
  defined by the Weyl structure is given in Proposition 5.1.10 of
  \cite{book}. This readily shows that the component of
  $\nabla^{\Cal V}s$ in $T^*M\otimes\Cal V^0/\Cal V^1$ is given by
  $\nabla\si+\partial(\mu)$. If $s$ is parallel, this vanishes
  identically, which implies the first claim.

  We have already noted above that $\si(x)=0$ implies that $\mu(x)$ is
  independent of the Weyl structure chosen. General facts about linear
  connections (see Section \ref{2.7} below) imply that $\nabla\si(x)$
  is independent of $\nabla$. Vanishing one jet of $\si$ in $x$ is of
  course equivalent to $\nabla\si(x)=0$ and hence to
  $\partial(\mu(x))=0$. We have noted above that $\partial$ is
  injective, so this is equivalent to $\mu(x)=0$. Assuming $\si(x)=0$,
  $\mu(x)$ coincides with the projection of $s(x)\in\Cal V^1_x$ to
  $\Cal V^1_x/\Cal V^2_x$, which implies the last claim.
\end{proof}

\begin{remark}\label{rem2.6} 
  One can say more about the bundle map $\partial$ from Proposition \ref{prop2.6}
  depending on the order of the first BGG operator determined by $V$. This is based
  on the developments in \cite{BCEG} for $|1|$--graded geometries, which use a
  different splitting operator, but can be easily adapted to the setting of the BGG
  splitting operator. If the first operator has order bigger than one, then
  $\partial:\Cal V^1/\Cal V^2\to T^*M\otimes\Cal V/\Cal V^1$ is an isomorphism of
  natural vector bundles. If the first BGG operator has order 1, then one can
  naturally decompose $T^*M\otimes\Cal V/\Cal V^1$ into the direct sum of
  $\im(\partial)$ and of $\ker(\partial^*)$, where
  $\partial^*: T^*M\otimes\Cal V/\Cal V^1\to \Cal V^1/\Cal V^2$ is induced in the
  obvious way by the action of $\frak g_1$ on $V $. The first BGG--operator is then given by
  the $\ker(\partial^*)$--component of $\nabla\si$ (which is the same for all
  Weyl--connections $\nabla$). Hence for a section in the kernel of the first BGG
  operator, $\nabla\si$ has values in $\im(\partial)$. In general, this component
  depends on the choice of the Weyl connection, but along the zero--locus of $\si$,
  it has invariant meaning.

  The methods of \cite{BCEG} lead to more general results (still in the $|1|$--graded
  case). If the order of the first BGG operator is $r$, then for $k\leq r$ vanishing
  of the $k$--jet of $\si=\Pi(s)$ of the BGG solution determined by a parallel
  section $s$ of a tractor bundle $\Cal V$ in a point $x$ is equivalent to the fact
  that $s(x)\in \Cal V^{k+1}_x\subset\Cal V_x$.  Moreover, if $k<r$ then assuming
  vanishing $k$--jet in $x$, the $k+1$--fold symmetrized covariant derivative in $x$
  can be computed algebraically from the class of $s(x)$ in
  $\Cal V^{k+1}_x/\Cal V^k_x$. The results of \cite{BCEG} have been extended to
  general parabolic geometries in \cite{Neusser:Prolon}, but in this case weighted
  jets and a concept of weighted order are required, so we do not go into this.
\end{remark}

\subsection{Defining sections}\label{2.7} 
We next discuss a generalization of the well--known concept of a
defining function (or, more generally, a defining density) for a
hypersurface to an analogous notion for the case of submanifolds of
higher codimension. We will efficiently use the jet information on
first BGG solutions discussed in Section \ref{2.6} by using them to
construct defining sections.

For a vector bundle $p:E\to M$ on a smooth manifold $M$, it is well known that two
linear connections on $E$ differ by a tensor field. Explicitly, for linear connections
$\nabla$ and $\tilde\nabla$ on $E$ there is a smooth section
$A\in\Ga(T^*M\otimes L(E,E))$ such that for any vector field $\xi\in\frak X(M)$ and
any section $\si\in\Ga(E)$, we get $\tilde\nabla_\xi \si=\nabla_\xi
\si+A(\xi)(\si)$.
In particular, this shows that for a point $x\in M$ such that $\si(x)=0$, we obtain
$\tilde\nabla_\xi \si(x)=\nabla_\xi \si(x)$, so we obtain a well defined map
$\nabla \si(x):T_xM\to E_x$. (This corresponds to the fact that the kernel of the
natural projection $J^1E\to E$ from the first jet prolongation of $E$ to $E$ is
naturally isomorphic to $T^*M\otimes E$.) Thus the following is well defined.

\begin{definition}\label{def2.7}
  Let $N\subset M$ be a smooth submanifold of codimension $k$ in a
  smooth manifold $M$ and let $p:E\to M$ be a smooth vector bundle of
  rank $k$. Then a local section $\si$ of $E$ defined in a
  neighborhood $U$ of a point $x\in N$ is called a \textit{local
    defining section} for $N$ if and only if for each $y\in U\cap N$
  we have that $\si(y)=0$ and the linear map $\nabla \si(y):T_yM\to E_y$ is
  surjective.
\end{definition}

Similarly to defining functions, we can use defining sections to
produce local coordinates around the submanifold $N$. 

\begin{prop}\label{prop2.7}
Let $N\subset M$ be a smooth submanifold of codimension $k$, $p:E\to
M$ a smooth vector bundle of rank $k$ and $\si\in\Ga(E)$ a defining
section for $N$ defined locally around a point $x\in N$. Then for any
local frame $\{\tau_1,\dots,\tau_k\}$ for $E$ defined on a
neighborhood of $x$, there is an open neighborhood $U$ of $x$ in $M$
contained in the domain of definition of the frame and a surjective
submersion $\pi:U\to U\cap N$ such that the map $\ps:U\to (U\cap
N)\x\Bbb R^k$ defined by $\ps(y):=(\pi(y),\si_1(y),\dots,\si_k(y))$ is
a diffeomorphism onto an open neighborhood of $(U\cap N)\x\{0\}$. Here
the functions $\si_i:U\to\Bbb R$ are the coordinate functions of $\si$
with respect to the frame $\{\tau_i\}$,
i.e.~$\si(y)=\sum_i\si_i(y)\tau_i(y)$ for all $y\in U$.
\end{prop}
\begin{proof}
  Let us start with chart $u:U\to u(U)\subset\Bbb R^{n-k}\x\Bbb R^k$, adapted to the
  submanifold $N$, which is defined on an open neighborhood $U$ of $x$ in $M$ on
  which the frame $\{\tau_i\}$ is defined. Projection onto the first
  ($(n-k)$--dimensional) factor in that chart defines a surjective submersion
  $\pi:U\to U\cap N$. The chosen frame then defines a trivialization
  $\ph:p^{-1}(U)\to U\x\Bbb R^k$ of $E$ over $U$. Moreover, we can define a linear
  connection $\nabla$ on $E|_U$ by declaring the elements of the frame to be
  parallel. By definition, for a vector field $\xi$, we get
  $\nabla_\xi \si=\sum_i(\xi\cdot \si_i)\tau_i$, so since $\si$ is a defining
  section, we see that the function $(\si_1,\dots,\si_k):U\to \Bbb R^k$ has
  surjective derivative along $U\cap N$. Since $\pi$ is a surjective submersion, this
  shows that $\psi:U\to (U\cap N)\x\Bbb R^k$ has surjective derivative along
  $U\cap N\subset U$, so the result follows from the inverse function theorem.
\end{proof}

\section{Parabolic compactifications related to $SO(p,q)\subset
  SL(p+q,\Bbb R)$}\label{3}
In this section we use the tools we have developed to study parabolic
compactifications obtained from the realization of $SO(p,q)$ as the
fixed point group of an involutive automorphism of $SL(p+q,\Bbb R)$ as
discussed in part (1) of Example \ref{ex2.3}. In particular, this
includes a parabolic compactification of the Riemannian symmetric
space $SO(p,q)/S(O(p)\x O(q))$, which we analyze in detail. 

\subsection{Orbits and infinitesimal transversals}\label{3.1} 
We have already noted in part (1) of Example \ref{ex2.3} that $H:=SO(p,q)$ acts with
finitely many orbits on each flag manifold of $G:=SL(p+q,\Bbb R)$. There we have also
noted that on Grassmannians, the orbits are determined by the rank and signature of
the restriction of the symmetric bilinear form defining $H$ to the subspace in
question and open orbits correspond to non--degenerate restrictions. On the
Grassmannian $Gr(i,\Bbb R^{p+q})$ we can thus index the orbits as $\Cal O_{(r,s)}$
where $0\leq r\leq p$ and $0\leq s\leq q$ and $r+s\leq i$ and the open orbits are the
ones for which $r+s=i$. We also see immediately that the closure of $\Cal O_{(r,s)}$
is the union of the orbits $\Cal O_{(r',s')}$ where $r'\leq r$ and $s'\leq s$. We
next determine the structure and the codimension of each orbit and the form of the
infinitesimal transversal as discussed in Section \ref{2.4}.

\begin{prop}\label{prop3.1}
  For $r,s$ as above, define $\nu:=i-r-s$, $\hat r=p-r-\nu$, and $\hat s=q-s-\nu$,
  and consider the orbit $\Cal O:=\Cal O_{(r,s)}\subset Gr(i,\Bbb R^{p+q})$. Further
  let $IGr(\nu,\Bbb R^{(p,q)})$ be the Grassmannian of isotropic subspaces of
  dimension $\nu$ in $\Bbb R^{(p,q)}$.  Then we have:

  (1) The orbit $\Cal O$ is non--empty if and only if $\nu\leq\text{min}(p,q)$ in 
  that case, its codimension  in $Gr(i,\Bbb R^{p+q})$ is $\nu(\nu+1)/2$.

  (2) There is an $H$--equivariant surjective submersion 
  $\Cal O\to IGr(\nu,\Bbb R^{(p,q)})$ whose fibers are isomorphic to the symmetric
  space $SL(p+q-2\nu,\Bbb R)/S(O(r,s)\x O(\hat r,\hat s))$.

  (3) Replacing $H$ by a conjugate subgroup $H_{\Cal O}$ as described in Section
  \ref{2.4}, there is a natural quotient homomorphism
  $H_{\Cal O}\cap P\to GL(\nu,\Bbb R)\x S(O(r,s)\x O(\hat r,\hat s))$. Thus the
  representation $S^2\Bbb R^{\nu*}$ of $GL(\nu,\Bbb R)$ gives rise to a
  representation of $H_{\Cal O}\cap P$, which is isomorphic to
  $(\frak g/\frak p)/(\frak h_{\Cal O}/(\frak h_{\Cal O}\cap\frak p))$.
\end{prop}
\begin{proof}
  Let $b$ be the bilinear form defining $H$ and let $W\subset\Bbb R^{p+q}$ be a
  subspace of dimension $i$ such that $b|_{W\x W}$ has signature $(r,s)$. Then the
  null space $W\cap W^\perp$ of $b|_{W\x W}$ has dimension $\nu$ and is totally
  isotropic for $b$. Since the maximal dimension of a totally isotropic subspace is
  $\text{min}(p,q)$ we conclude that $\Cal O_{(r,s)}$ is empty for
  $\nu>\text{min}(p,q)$. Otherwise, mapping $W$ to $W\cap W^\perp$ defines an
  $H$--equivariant map from $\Cal O_{(r,s)}$ to $IGr(\nu,\Bbb R^{(p,q)})$, which is
  easily seen to be smooth. Next, the sum $W+W^\perp$ has orthogonal space
  $W\cap W^\perp$ and thus is co--isotropic, so $b$ induces a duality between
  $W\cap W^\perp$ and $\Bbb R^{p+q}/(W+W^\perp)$. On the other hand, $b$ induces a
  non--degenerate bilinear form $\underline{b}$ on $(W+W^\perp)/(W\cap W^\perp)$
  which has signature $(p-\nu,q-\nu)$. By construction $W$ and $W^\perp$ descend to
  complementary subspaces in the quotient, on which the signature of $\underline{b}$
  is $(r,s)$ and $(\hat r,\hat s)$, respectively.

  Conversely, choosing a subspace $\Cal N\subset\Bbb R^{p+q}$ of dimension $\nu$
  which is totally isotropic for $b$, there is an induced bilinear form
  $\underline{b}$ on $\Cal N^\perp/\Cal N$ which is non--degenerate of signature
  $(p-\nu,q-\nu)$. Choosing a subspace of dimension $r+s$ in $\Cal N^\perp/\Cal N$,
  on which $\underline{b}$ has signature $(r,s)$, the pre--image in $\Cal N^\perp$ is
  a subspace of dimension $i$, which clearly lies in $\Cal O_{(r,s)}$. This shows
  that $\Cal O_{(r,s)}$ is non-empty if $\nu\leq\text{min}(p,q)$ and that our map
  $\Cal O_{(r,s)}\to IGr(\nu,\Bbb R^{(p,q)})$ is surjective. Since these are
  homogeneous spaces of $H$ and the map is $H$-equivariant, it must be the natural
  projection $H/K_1\to H/K_2$ for $K_1\subset K_2\subset H$, which is a
  submersion. The fiber of this map over $\Cal N$ is isomorphic to the space of those
  linear subspaces in $\Cal N^\perp/\Cal N\cong \Bbb R^{p+q-2\nu}$, on which
  $\underline{b}$ has signature $(r,s)$, which implies the claimed description as a
  symmetric space.

  Returning to a fixed subspace $W\in\Cal O_{(r,s)}$, the above considerations easily
  imply that we can find a basis $\{v_j\}$ for $\Bbb R^{p+q}$ which is adapted to the
  flag $W\cap W^\perp\subset W\subset W+W^\perp\subset\Bbb R^{p+q)}$ such that the
  matrix $(b(v_j,v_\ell))$ has the block form
  $\left(\begin{smallmatrix} 0 & 0 & 0 & \Bbb I\\ 0 & \Bbb I_{r,s}& 0 & 0\\ 0 & 0 &
      \Bbb I_{\hat r,\hat s} &0\\ \Bbb I& 0 & 0 & 0 \end{smallmatrix}\right)$.
  Here the blocks have sizes $\nu$, $r+s$, $\hat r+\hat s$, and $\nu$, $\Bbb I$
  denotes the identity matrix, and $\Bbb I_{r,s}$ is the diagonal matrix with $r$
  diagonal entries equal to $+1$ and $s$ entries equal to $-1$. (This also shows that
  any two subspaces for which the restriction of $b$ has signature $(r,s)$ are
  conjugate under the action of $H$.)

  Now take $H_{\Cal O}$ to be the orthogonal group corresponding to
  the above matrix. A simple direct computation then shows that, in a
  block form as above, the Lie algebra $\frak h_{\Cal O}$ consists of
  all matrices of the form
$$
\begin{pmatrix}
  A & K & L & M \\ E & B & -\Bbb I_{r,s}D^t\Bbb I_{\hat r,\hat s} &
  -\Bbb I_{r,s} K^t\\ F & D & C & -\Bbb I_{\hat r,\hat s}L^t \\ G &
  -E^t\Bbb I_{r,s} & -F^t\Bbb I_{\hat r,\hat s}& -A^t
\end{pmatrix}
\quad\text{\ with\ }\quad
\begin{matrix}
  B\in\frak o(r,s), C\in\frak
  o(\hat r,\hat s)\\
  G^t=-G, M^t=-M
\end{matrix}
$$
By definition $\frak h_{\Cal O}\cap\frak p$ corresponds to those matrices, for which
the four blocks in the lower left corner vanish. But these are exactly the matrices
which are block--upper triangular with respect to the finer block
decomposition. Hence $\frak h_{\Cal O}/(\frak h_{\Cal O}\cap\frak p)$ corresponds to
the four blocks in lower left corner, while $\frak g/\frak p$ is represented by
arbitrary matrices of size $(n-i)\x i$ in that part. But the only restriction on
these four blocks, implied by lying in $\frak h_{\Cal O}$, is that the $G$--block has
to be skew symmetric. As a vector space, the infinitesimal transversal
$(\frak g/\frak p)/(\frak h_{\Cal O}/(\frak h_{\Cal O}\cap\frak p))$ thus is
isomorphic to the space of symmetric matrices of size $\nu$. This shows that
$\Cal O_{(r,s)}$ has codimension $\nu(\nu+1)/2$ in $G/P$, which completes the proof
of parts (1) and (2).

Finally, the matrices which are strictly block--upper-triangular form an ideal in
$\frak h_{\Cal O}\cap\frak p$, with the quotient corresponding to the block diagonal
part. On the group level, this corresponds to the claimed quotient homomorphism in
(3). Explicitly, this sends a linear map that preserves each of the subspaces in the
chain $W\cap W^\perp\subset W\subset W+W^\perp$ to the induced maps on
$W\cap W^\perp$ and the quotients $W/(W\cap W^\perp)$ and $(W+W^\perp)/W$. Since
in the adjoint representation, the strictly block--upper--triangular matrices act
trivially on the $G$--block, we see that the natural action of $H_{\Cal O}\cap P$ on
the infinitesimal transversal descends to the quotient, with only the
$GL(\nu,\Bbb R)$--factor acting non--trivially. This completes the proof of (3).
\end{proof}

\begin{remark}\label{rem3.1}
  Let us briefly discuss the orbit structure for more general flag
  manifolds. Consider the case of two--step flags
  $W_1\subset W_2\subset\Bbb R^{(p,q)}$ of dimension $(i_1,i_2)$. The signatures
  $(r_1,s_1)$ and $(r_2,s_2)$ of the restrictions of $b$ to $W_1$ and $W_2$ satisfy
  $r_1\leq r_2$ and $s_1\leq s_2$ and these signatures are constant on
  $H$--orbits. Next, we have $W_2^\perp\subset W_1^\perp$ and the intersections
  $W_1\cap W_1^\perp$ and $W_2\cap W_2^\perp$ are the null spaces of the restrictions
  of $b$, which have dimension $\nu_j:=i_j-r_j-s_j$ for $j=1,2$. But at this point it
  is clear that, if both $\nu_1$ and $\nu_2$ are non--zero, then an additional
  invariant pops up: There is the natural subspace
  $W_1\cap W_2^\perp\subset W_1\cap W_1^\perp$ of vectors in the null space of
  $b|_{W_1\x W_1}$ which remain isotropic in $W_2$. The codimension $\ell$ of this
  subspace of course is preserved by the action of $H$. 

  There are evident restrictions on $\ell$. First, since
  $W_1\cap W_2^\perp\subset W_2\cap W_2^\perp$ we must have
  $\nu_1-\ell\leq\nu_2$. Second, vectors in
  $W_1\cap W_1^\perp\setminus W_1\cap W_2^\perp$ are isotropic but not orthogonal to
  $W_2$. Hence there is a subspace of dimension $2\ell$ in $W_2$, whose intersection
  with $W_1$ is complementary to $W_1\cap W_2^\perp$ in $W_1\cap W_1^\perp$ and on
  which $b$ has split signature $(\ell,\ell)$. This shows that $\ell\leq r_2-r_1$ and
  $\ell\leq s_2-s_1$. These are the only restrictions on $\ell$, however.

  Indeed it is an exercise in linear algebra to show that fixing $(r_j,s_j)$ for
  $j=1,2$ and $\ell$ such that the above restrictions are satisfied, one can find a
  basis for $\Bbb R^{p+q}$ adapted to $W_1$ and $W_2$ for which the inner product has
  a standard form. Hence these parameters completely determine the orbits. Moreover,
  there is a neighborhood of the orbit corresponding to $(r_j,s_j,\ell)$ on which the
  corresponding parameters satisfy $r_j'\geq r_j$, $s_j'\geq s_j$ and
  $\ell'\leq\ell$. Using this, one concludes that the closure of the orbit determined
  by $(r_j,s_j,\ell)$ consists of all orbits corresponding to $(r'_j,s'_j,\ell')$
  such that $r_j'\leq r_j$, $s_j'\leq s_j$ and $\ell'\geq\ell$.
\end{remark}

\subsection{Orbit closures via solutions of BGG operators}\label{3.2}
We continue the study of the orbits of $H=SO(p,q)$ in a Grassmannian
$G/P\cong Gr(i,\Bbb R^n)$, where $n=p+q$. We next describe certain unions of such
orbits as the zero sets of solutions of appropriate first BGG operators. In
particular, this provides a description of all orbit closures for the parabolic
compactification of $SO(p,q)/S(O(p)\x O(q))$ discussed in Section \ref{2.5} as zero
sets, see Remark \ref{rem3.4} below.

Recall the construction of the two tautological vector bundles $E$ and $F$ on the
Grassmannian $Gr(i,\Bbb R^n)$. These fit into an exact sequence of the form
$0\to E\to \Bbb R^n\to F\to 0$, where $\Bbb R^n$ indicates a trivial bundle. Viewing
a point in $Gr(i,\Bbb R^n)$ as a linear subspace $V\subset\Bbb R^n$, the fibers of
$E$ and $F$ over $V$ are $V$ and $\Bbb R^n/V$, respectively. By definition, this is
the standard tractor bundle $\Cal T$ for the flat Grassmannian structure with its
$P$--invariant filtration structure, in the trivialization described in Section
\ref{2.6}. In particular, the canonical completely reducible quotient of $\Cal T$ is
the anti--tautological bundle $F$. For our choice of parabolic subgroup $P$, the
reductive Levi--component $G_0$ is isomorphic to $S(GL(i,\Bbb R)\x GL(n-i,\Bbb R))$,
and the bundles $E$ and $F$ correspond to the standard representations of the two
factors. Thus we see that all the completely reducible natural bundles on
$Gr (i, \Bbb R^{p+q})$ can be built up from $E$, $F$, and their duals via tensorial
constructions.

At this point, we need a bit of background from representation theory. Viewing the
bilinear form $b$ we have used to define $H$ as a linear isomorphism
$\Bbb R^n\to \Bbb R^{n*}$, we can form the induced linear isomorphism
$\La^kb:\La^k\Bbb R^n\to \La^k\Bbb R^{n*}$. On the other hand, $b$ induces a
symmetric bilinear form $\tilde b$ on $\La^k\Bbb R^n$ defined by
$$
\tilde b(v_1\wedge\dots\wedge
v_k,w_1\wedge\dots\wedge w_k):=\det((b(v_i,w_j))_{i,j=1,\dots,k}),  
$$ 
which can be viewed as an element of $S^2(\La^k\Bbb R^{n*})$. Unless
$k=1$ or $k\geq n-1$, this is not an irreducible representation of
$SL(n,\Bbb R)$. There is however a maximal irreducible component,
which we denote by $\ocirc^2(\La^k\Bbb R^{n*})$, whose highest weight
equals twice the highest weight of the irreducible representation
$\La^k\Bbb R^{n*}$.

\begin{lemma}\label{lem3.2}
Identifying $L(\La^k\Bbb R^n,\La^k\Bbb R^{n*})$ with
$\otimes^2\La^k\Bbb R^{n*}$ the element $\La^kb$ coincides with
$\tilde b\in S^2(\La^k\Bbb R^{n*})$. Moreover, this element is
automatically contained in the irreducible component
$\ocirc^2(\La^k\Bbb R^{n*})$. 
\end{lemma}
\begin{proof}
The map $\Bbb R^n\to\Bbb R^{n*}$ induced by $b$ sends $v$ to the
linear functional $b(v,\_)$. Hence the induced map on the $k$th
exterior powers sends $v_1\wedge\dots\wedge v_k$ to
$b(v_1,\_)\wedge\dots\wedge b(v_k,\_)$. Evaluating this functional on
$w_1\wedge\dots\wedge w_k$ we obtain $\tilde b(v_1\wedge\dots\wedge
v_k,w_1\wedge\dots\wedge w_k)$, which proves the first claim. 

For the second claim, we observe that there is a homomorphism of representations of
$SL(n,\Bbb R)$ mapping
$S^k(L(\Bbb R^n,\Bbb R^{n*}))\to L(\La^k\Bbb R^n,\La^k\Bbb R^{n*})$. For
$f_1,\dots,f_k\in L(\Bbb R^n,\Bbb R^{n*})$, the element
$f_1\vee\dots\vee f_k\in S^k(L(\Bbb R^n,\Bbb R^{n*}))$, defines a linear map
$\La^k\Bbb R^n\to \La^k\Bbb R^{n*}$ via
$$
(f_1\vee\dots\vee f_k)(v_1\wedge\dots\wedge
v_k):=\tfrac{1}{k!}\textstyle\sum_{\si\in\frak
  S(k)}f_{\si_1}(v_1)\wedge\dots\wedge f_{\si_k}(v_k).
$$ 
From this explicit formula, it follows readily this restricts to a homomorphism
$S^k(S^2\Bbb R^{n*})\to S^2(\La^k\Bbb R^{n*})$ and by construction $\La^kb$ lies in
the image of this homomorphism. In terms of Young diagrams, irreducible components of
$S^k(S^2\Bbb R^{n*})$ are obtained by arranging $k$ copies of a horizontal pair of
boxes into a Young diagram. This implies that the highest weight of
$\ocirc^2(\La^k\Bbb R^{n*})$ occurs among these weights (as a ``tower'' of $k$
horizontal pairs). This young diagram corresponds to the smallest among all the
highest weights of irreducible components of $S^k(S^2\Bbb R^{n*})$. Thus no other
irreducible component of $S^2(\La^k\Bbb R^{n*})$ can be contained in the image of our
homomorphism. Hence this image coincides with $\ocirc^2\La^k\Bbb R^{n*}$, which
completes the proof.
\end{proof}

As we have observed above, irreducible bundles on the Grassmannian
come from representations of $S(GL(i,\Bbb R)\x GL(n-i,\Bbb R))$ with
the tautological bundle $E$ corresponding to the standard
representation of the first factor. In particular, this implies that
we can form the bundles $\ocirc^2(\La^kE^*)$ for $k=1,\dots,i$ and for
$k=1,i-1,i$, this bundle coincides with $S^2(\La^kE^*)$.

\begin{thm}\label{thm3.2}
Consider the decomposition $Gr(i,\Bbb R^{p+q})=\cup_{r,s}\Cal
O_{(r,s)}$ into orbits of the group $H=SO(p,q)$. Then for each
$k=1,\dots,i$, there is a section $\si_k\in\Ga(\ocirc^2(\La^kE^*))$,
which lies in the kernel of the first BGG operator naturally defined
on that bundle, whose zero set is the union of all those $H$--orbits
$\Cal O_{(r,s)}$ for which $r+s<k$. The relevant first BGG operator is
of order one for $k<i$ and of order $3$ for $k=i$.
\end{thm}
\begin{proof}
  From the filtration structure of the standard tractor bundle $\Cal
  T$ described above, it readily follows that the dual bundle $\Cal
  T^*$ has $E^*$ as its canonical completely reducible quotient. Hence
  for $S^2\Cal T^*$, the completely reducible quotient is
  $S^2E^*$. Now $b\in S^2\Bbb R^{n*}$ corresponds to a parallel
  section $s=s_1$ of $S^2\Cal T^*$, which projects to
  $\si_1=\Pi(s_1)\in\Ga(S^2E^*)$. From the above description of the
  trivialization of $\Cal T$ it readily follows that the value of
  $\si$ in a point $V\in Gr(i,\Bbb R^{p+q})$ is simply given by the
  restriction of $b$ to the fiber of $E$ over $V$, which coincides with
  $V$. Hence we see that $V\in\Cal O_{(r,s)}$ if and only if
  $\si_1(V)$ has rank $r+s$ and signature $(r,s)$ as a bilinear
  form. On the other hand, $\si_1$ lies in the kernel of the first BGG
  operator by Theorem \ref{thm2.6}.

  Now for $k=2,\dots,i$ we can proceed similarly starting with
  $\La^kb$, which by Lemma \ref{lem3.2} lies in $\ocirc^2(\La^k\Bbb
  R^{n*})$. Hence it gives rise to a parallel section
  $s_k\in\Ga(\ocirc^2(\La^k\Cal T^*))$, which in turn projects onto a
  section $\si_k:=\Pi(s_k)$ of the irreducible quotient, which lies in
  the kernel of the corresponding first BGG operator. From the
  description of completely reducible quotients in terms of highest
  weights and since $k\leq i$, it readily follows that this quotient
  equals $\ocirc^2(\La^kE^*)$. Moreover, the projection $\Pi$ is again
  given by mapping a bilinear form to its restriction to the fibers of
  $\La^kE$. Hence we conclude that $\si_k(V)$ is the restriction of
  $\La^kb$ to $\La^kV$, so it coincides with $\La^k(\si_1(V))$, which
  by Lemma \ref{lem3.2} lies in $\ocirc^2(\La^kV^*)$.

  But now consider $\si_1(V)$ as a map $V\to V^*$ and let $W\subset V^*$ be its
  image. Then we can factorize our map as $V\to W\hookrightarrow V^*$ so by
  functoriality, $\La^k(\si_1(V))$ factorizes as
  $\La^k V\to \La^kW\hookrightarrow \La^kV^*$. This shows that
  $\La^k(\si_1(V))=\si_k(V)$ vanishes if and only if the rank of $\si_1(V)$ is less
  than $k$, i.e.~iff $V$ lies in an orbit $\Cal O_{(r,s)}$ such that $r+s<k$. The
  order of the relevant first BGG operators can be read off the highest weight of the
  inducing representations, compare with \cite{BCEG}.
\end{proof}

\subsection{A slice theorem}\label{3.3}
We next derive a description of a neighborhood of one of the orbits
$\Cal O_{(r,s)}\subset Gr(i,\Bbb R^{p+q})$. In view of the description of the
infinitesimal transversal in Proposition \ref{prop3.1}, it is visible what the best
possible result would be, c.f.~Remark \ref{rem2.4}. Putting $\nu=i-r-s$, the
infinitesimal transversal can be identified with $S^2\Bbb R^{\nu*}$, with the action
of the isotropy group coming from the natural representation of $GL(\nu,\Bbb R)$ on
that space. Hence the orbits of the isotropy group on the infinitesimal transversal
are determined by rank and signature, and following the philosophy of slice theorems,
the optimal result to expect would be a parallel description of a neighborhood of
$\Cal O_{(r,s)}$ in the Grassmannian. The following result is the key step to showing
that such a description is available locally.

\begin{thm}\label{thm3.3}
  Consider the section $\si_1\in\Ga(S^2E^*)$ from Theorem \ref{thm3.2}, a
  pair $(r,s)$ with $0\leq r\leq p$, $0\leq s\leq q$ and $r+s<i$, and a point
  $x\in\Cal O_{(r,s)}$. Then there are open neighborhoods $U$ of $x$ in $M$ and
  $\underline{U}\subset U\cap\Cal O_{(r,s)}$ of $x$ in $\Cal O_{(r,s)}$ and there is
  a smooth subbundle $\tilde E\subset E|_{U}$ of rank $\nu:=i-r-s$ such that
  \begin{itemize}
  \item for each $y\in U$ the null--space of $\si_1(y)$ is contained
    in $\tilde E_y\subset E_y$
  \item the obvious projection of $\si_1$ to a section of $S^2\tilde E^*|_U$ is a
    defining section for $\underline{U}$.
  \end{itemize}
\end{thm}
\begin{proof}
  From Proposition \ref{prop3.1}, we know that $\Cal O_{(r,s)}$ has codimension
  $\nu(\nu+1)/2$ in the Grassmannian, so $\nu$ is the right rank for a subbundle
  $\tilde E$ to have a chance for a defining section of $S^2\tilde E^*$. Since
  $\Cal O_{(r,s)}$ is an initial submanifold in $G/P$, there is a connected open
  neighborhood $\underline{U}$ of $x$ in $\Cal O_{(r,s)}$, which is a true
  submanifold of $M$. (Take an adapted chart for the initial submanifold centered at
  $x$ and let $\underline{U}$ be the pre--image of the connected component of $x$ in
  the image of $\Cal O_{(r,s)}$ in that chart.)  Next, choose a linear subspace
  $W_x\subset E_x$ of dimension $r+s$ on which $\si_1(x)$ is non--degenerate and has
  signature $(r,s)$. This can be extended to a smooth subbundle $W\subset E$ on some
  connected open neighborhood $U$ of $x$ in $G/P$, which can be assumed to contain
  $\underline{U}$. Shrinking $U$ and $\underline{U}$ if necessary, we may assume that
  for each $y\in U$, the bilinear form $\si_1(y)$ is non--degenerate of signature
  $(r,s)$ on $W_y$. In particular, this implies that
  $U\subset \cup_{r'\geq r,s'\geq s}\Cal O_{(r',s')}$.

  Now for each $y\in U$, we define $\tilde E_y$ to be the space of all $v\in E_y$
  such that $\si_1(y)(v,w)=0$ for all $w\in W_y$. By non--degeneracy of $\si_1(y)$ on
  $W_y$, each of these spaces has dimension $\nu$ and, of course, it always contains
  the null space of $\si_1(y)$. We claim that, possibly shrinking $U$ and
  $\underline{U}$ further, there is a local smooth frame for $\tilde E$, so this is a
  smooth subbundle of $E|_U$. Indeed, take smooth sections
  $\tau_1,\dots,\tau_{r+s}\in\Ga(W)$ that form a local frame for $W$ and extend them
  by smooth sections $\tau_{r+s+1},\dots,\tau_i\in\Ga(E)$ to a local frame for
  $E$. Then non degeneracy of $\si_1$ on $W$ implies that for each $j=r+s+1,\dots,i$,
  one can add a smooth linear combination of $\tau_1,\dots,\tau_{r+s}$ to $\tau_j$ in
  such a way that the result has values in $\tilde E$.  Of course, the resulting
  sections then have to form a smooth local frame for $\tilde E$.

  Having constructed $\tilde E\subset E|_U$, there is a natural
  projection $q:S^2E^*|_U\to S^2\tilde E^*$, obtained by restricting
  bilinear forms to taking entries from the fibers of $\tilde E$. For
  $y\in \underline{U}\subset\Cal O_{(r,s)}$, we know that $\tilde E_y$ has
  dimension $\nu$ and contains the null--space of $\si_1(y)$, so it
  has to coincide with that null--space. Hence $q\o\si_1$ vanishes
  along $\underline{U}$ and we only have to verify that $\nabla (q\o
  \si)(y):T_y(G/P)\to S^2\tilde E_y^*$ is surjective for some
  connection $\nabla$ on $S^2\tilde E^*$ and each $y\in \underline{U}$ to complete
  the proof. 

  Suppose that $\nabla$ is a linear connection on $E$. Using the decomposition
  $E|_U=\tilde E\oplus W$, we obtain an induced linear connection $\tilde\nabla$ on
  $\tilde E$. For $v\in\Ga(\tilde E)$ and $\xi\in\frak X(U)$, we simply define
  $\tilde\nabla_\xi v$ as the $\tilde E$--component of $\nabla_\xi v$. Then consider
  the induced connections on $S^2E^*$ and $S^2\tilde E^*$, respectively, which we
  also denote by $\nabla$ and $\tilde\nabla$. For $v,w\in\Ga(\tilde E)$, we thus have
$$
(\tilde\nabla_\xi(q\o\si))(v,w)=\xi\cdot (q\o\si)(v,w)-(q\o\si)(\tilde\nabla_\xi
v,w)-(q\o\si)(v,\tilde\nabla_\xi w).
$$
By definition, $(q\o\si)(v,w)=\si(v,w)$ and, along $\underline{U}$, $\si$ vanishes
identically upon insertion of either $v$ or $w$ by construction. Hence we see that
for $y\in \underline{U}$ and each $\xi\in\frak X(U)$, we get
$\tilde\nabla_\xi (q\o\si)(y)=q(\nabla_\xi\si(y))$.

Taking $\nabla$ to be the Weyl connection defined by some Weyl structure, we can
compute $\nabla\si$ using Proposition \ref{prop2.6}. For the Grassmannian, the
tangent bundle $T(G/P)$ is isomorphic to $E^*\otimes F$, so
$T^*(G/P)\cong E\otimes F^*$. On the other hand, for $\Cal V=S^2\Cal T^*$, we have
$\Cal V/\Cal V^1\cong S^2E^*$ and $\Cal V^1/\Cal V^2\cong E^*\otimes F^*$. Hence in
our case the bundle map $\partial$ from Proposition \ref{prop2.6} maps
$F^*\otimes E^*$ to $F^*\otimes E\otimes S^2E^*$. By $P$--equivariancy, this must
coincide (up to a non--zero constant factor) with tensorizing with the
$\id_E\in E^*\otimes E$ and then symmetrizing. Viewing $\mu\in F^*\otimes E^*$ as a
linear map $F\to E^*$ and $\partial(\mu)$ as a map $F\otimes E^*\to S^2E^*$ we
conclude that $\partial(\mu)(f\otimes\lambda)=\mu(f)\vee\la$.

Now in a point $y\in\underline{U}$, we know that $\tilde E_y$ is the null--space of
$\si_1(y)$. Recall that $\si_1=\Pi(s_1)$ for a parallel section
$s_1\in\Ga(S^2\Cal T^*)$, so $\si_1(y)=s_1(y)|_{E_y\x E_y}$. Restricting $s_1(y)$ to
$\Cal T_y\times\tilde E_y$, the result vanishes upon insertion of an element of
$E_y$ in the first factor, so this factors to a map $F_y\x \tilde E_y\to\Bbb R$.
By construction, this map has to coincide with the restriction of
$\mu(y):F_y\x E_y\to\Bbb R$ to $F_y\x \tilde E_y$.  Non--degeneracy of $s_1(y)$ then
implies that this induces a surjection $F_y\to \tilde E_y^*$. But this implies that
the restriction of $q\o\partial(\mu):F_y\otimes E^*_y\to S^2\tilde E^*_y$ to
$F_y\otimes\tilde E^*_y$ is onto, which completes the proof.
\end{proof}

This easily leads to a full description of the local structure around
each of the orbits $\Cal O_{(r,s)}$. 
\begin{cor}\label{cor3.3}
Each of the orbits $\Cal O_{(r,s)}$ is an embedded submanifold of
$G/P$. Moreover, for each point $x\in\Cal O_{(r,s)}$, there is an open
neighborhood $U$ of $x$ in $G/P$ and a diffeomorphism $\ph$ from $U$
onto an open neighborhood of $(U\cap\Cal O_{(r,s)})\x\{0\}$ in
$(U\cap\Cal O_{(r,s)})\x S^2\Bbb R^{(n-r-s)*}$ such that
$U\subset\cup_{r'\geq r,s'\geq s}\Cal O_{(r',s')}$ and a point $y\in
U$ lies in $\Cal O_{(r',s')}$ if and only if the second component of
$\ph(y)$ has signature $(r'-r,s'-s)$.
\end{cor}
\begin{proof}
  We take $\underline{U}\subset U$ and $q\o\si_1$ as in Theorem \ref{thm3.3} and then
  apply Proposition \ref{prop2.7} using a local frame of $S^2\tilde E^*$ determined
  by a local frame $\tau_1,\dots,\tau_\nu$ ($\nu=n-r-s$) for $\tilde E$. Possibly
  shrinking $U$ and $\underline{U}$, this gives a diffeomorphism from $U$ to an open
  neighborhood of $\underline{U}\x\{0\}$ in $\underline{U}\x S^2\Bbb R^{\nu*}$.
  Moreover, the coordinate functions of $q\o\si_1$ with respect to that frame are
  simply the functions $\si_1(\tau_a,\tau_b)$, so the rank and signature of the
  resulting symmetric matrix in a point $y$ coincides with the rank and signature of
  $\si_1(y)$ on $\tilde E_y$. By construction, this signature is $(p',q')$ if and
  only if the signature of $\si_1(y)$ on $E_y$ is $(p'+r,q'+s)$. This shows that
  $\underline{U}=U\cap\Cal O_{(r,s)}$, so in particular, $\Cal O_{(r,s)}$ is an
  embedded submanifold, and the characterization of $U\cap\Cal O_{(r',s')}$ follows
  easily.
\end{proof}

\subsection{A natural defining density for the largest boundary
  component}\label{3.4} 

Consider an orbit $\Cal O_{(r,s)}\subset Gr(i,\Bbb R^n)$ with
$r+s=i-1$.  Theorem \ref{thm3.3} produces (locally) a section of
$S^2\tilde E^*$ that is a local defining section for $\Cal
O_{(r,s)}$. But in this case $\tilde E$ is a line bundle and therefore
so is $S^2\tilde E^*$. Thus we obtain an analog of a defining density
as discussed on p.~52 of \cite{Proj-comp}. As a final step in the
discussion of the parabolic compactifications related to
$SO(p,q)\subset SL(p+q,\Bbb R)$, we show that, for these largest
non--open orbits, we also get a  defining density that is natural. In Theorem
\ref{thm3.2}, we have constructed a section $\si_i$ of the bundle
$\ocirc^2(\La^iE^*)$. Now since $E$ has rank $i$, $\La^iE^*$ is a line
bundle, so $\ocirc^2(\La^iE^*)=S^2(\La^iE^*)$ is a line bundle,
too. Since we are dealing with a $|1|$--graded geometry here, any
natural line bundle is a density bundle.

\begin{prop}\label{prop3.4}
  The section $\si_i$ of the density bundle $L:=S^2(\La^iE^*)$ is a
  defining density for each of the orbits $\Cal O_{(r,s)}\subset
  Gr_i(\Bbb R^{p+q})$ with $r+s=i-1$.  
\end{prop}
\begin{proof}
  We already know that $r+s=i-1$ implies that $\Cal O_{(r,s)}$ is an
  embedded hypersurface in the Grassmannian, and by Theorem
  \ref{thm3.2}, $\si_i$ vanishes along $\Cal O_{(r,s)}$. Thus it
  remains to verify that $\nabla\si_i$ is non--vanishing along $\Cal
  O_{(r,s)}$. To see this, we have to analyze the canonical
  $P$--invariant filtration $\{V^i\}$ of the representation
  $V:=\ocirc^2(\La^i\Bbb R^{n*})$ of $G$ respectively the filtration
  $\{\Cal V^i\}$ of the corresponding tractor bundle $\Cal V$. We
  already know that $\Cal V/\Cal V^1\cong L$ which was used to obtain
  $\si_i$ from the parallel section $s_i\in\Ga(\Cal V)$. Now
  Proposition \ref{prop2.6} shows that for a point $x\in G/P$
  simultaneous vanishing of $\si_i(x)$ and $\nabla\si_i(x)$ are
  equivalent to the fact that $s_i(X)\in\Cal V^2_x$. We have also
  observed that the first BGG--operator in our case has order $3$,
  which by Remark \ref{rem2.6} implies that $\Cal V^1/\Cal V^2\cong
  T^*(G/P)\otimes L$.

  On the other hand, consider the natural filtration of the tractor bundle
  $\Cal W:=\La^i\Cal T^*$. This is induced by inserting elements of the subbundle $E$
  into multilinear maps. In particular, $\Cal W^1$ consists of those maps which
  vanish under insertion of $i$ elements in $E$, which explains the isomorphism
  $\Cal W/\Cal W^1\cong\La^i E^*=L$. Likewise, $\Cal W^2\subset\Cal W^1$ consists of
  maps which vanish upon insertion of $i-1$ elements of $E$, so
  $\Cal W^1/\Cal W^2\cong F^*\otimes\La^{i-1}E^*$ and this has rank $(n-i)i$. Hence
  if we look at the natural filtration of $S^2\Cal W^*$, the iterated quotients of
  filtration components in the first two steps are given by $S^2L$ and
  $L\otimes F^*\otimes \La^{i-1}E^*$, respectively. It is easy to see that the latter
  bundle is isomorphic to $T^*(G/P)\otimes L$.

  Comparing the statements of the last two paragraphs, we conclude that the subbundle
  $\Cal V\subset S^2\Cal W$ has the property that $\Cal V^1/\Cal V^2$ surjects onto
  $(S^2\Cal W)^1/(S^2\Cal W)^2$. Otherwise put, if at some point $x$ we have
  $\si_i(x)=0$ and $\nabla\si_i(x)=0$, then, viewed as a map
  $\La^i\Cal T_x\to\La^i\Cal T^*_x$, $s_i(x)$ has the property that applying it to a
  wedge products of $i-1$ Elements of $E_x$ and one element of $F_x$, the result
  vanishes upon insertion of $i$ elements of $E_x$. But this implies that the
  restriction of $s_1(x)$ to $E_x$ has rank less than $i-1$. Indeed, if this rank is
  at least $i-1$, then we can choose a basis $\{e_1,\dots,e_i\}$ for $E_x$ such that
  $s_1(x)(e_i,e_j)$ equals $0$ if $i=1$ or $i\neq j$ and $1$ for
  $i=j>1$. Non-degeneracy of $s_1(x)$ then shows that there must be an element
  $f\in F_x$ such that $s_1(x)(e_1,f)\neq 0$. But then by construction
  $s_i(x)(f\wedge e_2\wedge \dots\wedge e_p,e_1\wedge\dots\wedge e_p)\neq 0$.
\end{proof}

\begin{remark}\label{rem3.4} 
  Let us specialize the results of this Section to the case of the Riemannian
  symmetric space $H/K:=SO(p,q)/S(O(p)\x O(q))$, which we know can be identified with
  the open orbit $\Cal O_p:=\Cal O_{(p,0)}$ in $Gr(p,\Bbb R^{p+q})$. The closure
  $\overline{H/K}$ can be written as $\cup_{j\leq p}\Cal O_j$, where we briefly write
  $\Cal O_j$ for $\Cal O_{(j,0)}$. From Proposition \ref{prop3.1}, we know that
  $\Cal O_j$ has codimension $(p-j)(p-j+1)/2$, and for each $j$ we get
  $\overline{\Cal O_j}=\cup_{i\leq j}\Cal O_i$. Thus Theorem \ref{thm3.2} shows that
  the orbit closure $\overline{\Cal O_j}$ coincides with the intersection of the zero
  locus of $\si_{p-j+1}$ with $\overline{H/K}$. In particular, the zero locus of
  $\si_p$ coincides with the closure of $\Cal O_{p-1}$, and by Proposition
  \ref{prop3.4}, $\si_p$ is a defining density locally around each point of that
  orbit.

  Also, the slice theorem given in Corollary \ref{cor3.3} takes a particularly nice
  form for this example. For each $\nu$, we can consider the space
  $S^2_{>0}\Bbb R^{\nu*}$ of positive definite symmetric $\nu\x \nu$--matrices over
  $\Bbb R$, which can be identified with the symmetric space $GL(\nu,\Bbb
  R)/O(\nu)$.
  The closure of $S^2_{>0}\Bbb R^{\nu*}$ in the space of all symmetric matrices is
  the space $S^2_{\geq 0}\Bbb R^{\nu*}$ of positive semi--definite matrices. This can
  be viewed as a ``local compactification'' of $GL(\nu,\Bbb R)/O(\nu)$ in the sense
  that for any compact neighborhood $W$ of $0$ in $S^2\Bbb R^{\nu*}$, the
  intersection $W\cap S^2_{\geq 0}\Bbb R^{\nu*}$ is a compactification of
  $W\cap S^2_{>0}\Bbb R^{\nu*}$. Now Corollary \ref{cor3.3} says that locally around
  a point in $\Cal O_j$, the compactification $\overline{H/K}$ looks like the product
  of $\Cal O_j$ with this local compactification of $GL(p-j,\Bbb R)/O(p-j)$.
\end{remark}

\section{Parabolic compactifications related to 
$SO(n,\Bbb C)\subset SO_0(n,n)$}\label{4} 
 
As in part (2) of Example \ref{ex2.3}, we consider a non--degenerate complex bilinear
form $b$ on $\Bbb C^n\cong(\Bbb R^{2n},J)$. There we have seen that the imaginary
part of $b$ defines a split--signature inner product $\langle\ ,\ \rangle$ on
$\Bbb R^{2n}$. This gives rise to an inclusion
$SO(n,\Bbb C)\hookrightarrow SO_0(n,n)$ as the fixed point subgroup of the involutive
automorphism $A\mapsto -JAJ$. In this section
we study the resulting parabolic compactifications, with an emphasis on the case of
the Riemannian symmetric space $SO(n,\Bbb C)/SO(n)$ of the non--compact type.

\subsection{Orbits and infinitesimal transversals}\label{4.1}
The generalized flag manifolds of $G:=SO_0(n,n)$ can be realized as the spaces of
isotropic flags in the standard representation $\Bbb R^{(n,n)}$. In addition, one has
to take into account here that in the case of maximally isotropic subspaces
(i.e.~those of dimension $n$), one has to distinguish between self--dual and
anti--self--dual subspaces, since (anti--)self--duality is preserved by the action of
$G$. As before, we will mainly discuss the case of isotropic Grassmannians. In
contrast to Section \ref{3.1}, not even these behave uniformly, but additional
complications arise in the non--maximal case.

Given a linear subspace $V\subset\Bbb R^{2n}$, which is isotropic for the imaginary
part of $b$, one can of course look at the restriction of the real part of $b$ to
$V$, which defines a symmetric bilinear form on $V$. The rank and signature of this
bilinear form are evidently preserved under the action of $H:=SO(n,\Bbb C)$, so they
are basic invariants of the $H$--orbit determined by $V$. In the case that $V$ has
the maximal possible dimension $n$, then these data determine the orbit and also the
(anti--)self--duality properties of $V$:
\begin{lemma}\label{lem4.1}
  Let $V\subset\Bbb C^n=\Bbb R^{2n}$ be a real linear subspace of real dimension $n$,
  which is isotropic for the imaginary part $\langle\ ,\ \rangle$ of $b$ and such
  that $\Re(b)|_V$ has signature $(r,s)$ with $r+s\leq n$. Then $\nu:=n-r-s$ is even,
  say $\nu=2k$, and there is is a complex basis $\{z_1,\dots,z_n\}$ for $\Bbb C^n$
  with respect to which $b$ has the block--matrix representation
  $\left(\begin{smallmatrix} 0 & 0 & \Bbb I\\ 0 & \Bbb I_{r,s} & 0 \\ \Bbb I & 0 &
      0 \end{smallmatrix}\right)$
  with blocks of size $k$, $r+s$, and $k$, such that $V$ is the real span of the
  vectors $z_j$ for $j=1,\dots,k+r+s$ and $iz_j$ for $j=1,\dots,k$.

  Moreover, choosing the orientation of $\Bbb R^{2n}$ appropriately, $V$ is
  self--dual if $n-s$ is even and anti--self--dual if $n-s$ is odd.
\end{lemma}
\begin{proof}
  Let us denote by $J$ the complex structure on $\Bbb C^n$, by $\perp$ the
  orthocomplement with respect to $\langle\ ,\ \rangle$ and by $\perp_b$ the
  orthocomplement with respect to $b$. Then the real part of $b$ can be written as
  $(v,w)\mapsto\langle v,Jw\rangle$, so the null space of its restriction is
  $W:=J(V)\cap V^\perp=J(V)\cap V$, where we have used that $V$ is maximally
  isotropic in the last step. Since $W$ evidently is a complex subspace of $V$, its
  real dimension $\nu$ has to be even. Putting $\nu=2k$, we choose a complex basis
  $\{z_1,\dots,z_k\}$ for $W$. On the other hand, we conclude that $W^{\perp_b}$ is a
  complex subspace of $\Bbb C^n$ which has complex dimension $n-k$ and contains $V$.

  Now $b$ descends to a non--degenerate complex bilinear form $\underline{b}$ on the
  complex vector space $W^{\perp_b}/W$, which has complex dimension $n-2k$. Since
  $W=V\cap J(V)$, the image $\underline{V}$ of $V\subset W^{\perp_b}$ in this
  quotient has to be a totally real subspace of real dimension $n-k$. Moreover, the
  restriction of $\underline{b}$ to this subspace has to have signature $(r,s)$ so
  $\underline{b}|_{\underline{V}}$ is non--degenerate. Choose a real orthonormal
  basis of $\underline{V}$ and pre--images $z_{k+1},\dots,z_{k+r+s}$ of the basis
  elements in $V$. Then by construction these vectors descend to a complex basis of
  $W^{\perp_b}/W$, so they span a complex subspace $\tilde V\subset W^{\perp_b}$ such
  that $W^{\perp_b}=W\oplus\tilde V$.

  Finally, we consider $\tilde V^{\perp_b}\subset\Bbb C^n$. This is a complex
  subspace of complex dimension $n-2k$, on which $b$ is non--degenerate, and which
  contains the subspace $W$ which is isotropic for $b$. Thus one can find a complex
  subspace $\tilde W\subset\tilde V^{\perp_b}$ which also is isotropic for $b$ and
  complementary to $W$ in there. So $b$ identifies $\tilde W$ with dual $W^*$ and
  hence we can find a complex basis $\{z_{n-k+1},\dots,z_n\}$ for $\tilde W$ such
  that $b(z_j,z_{n-k+\ell})=\delta_{j\ell}$ for $j,\ell=1,\dots,k$. By construction
  $\{z_1,\dots,z_n\}$ is a complex basis of $\Bbb C^n$ which has all required
  properties.

  To prove the statement about (anti--)self--duality, we fix the orientation of
  $\Bbb R^{2n}=\Bbb C^n$ in such a way that for any complex basis $\{z_1,\dots,z_n\}$
  the real basis $\{z_1,\dots,z_n,iz_1,\dots,iz_n\}$ has positive orientation. Taking
  the basis $\{z_1,\dots,z_n\}$ adapted to $V$ as above, we have to compute the
  Hodge--$*$ of
$$
\be:=z_1\wedge\dots\wedge z_{n-k}\wedge iz_1\wedge\dots \wedge iz_k,
$$ 
and we use the standard formula $\al\wedge *\be=\langle\al,\be\rangle\vol$. Now up to
sign, $\be$ is one of the elements of the basis for $\La^n_{\Bbb R}\Bbb R^{2n}$
induced by the basis $\{z_j,iz_j\}$. There is just one other element $\al$ in this
basis for which $\langle\al,\be\rangle$ is non--zero, namely the wedge product of the
elements $iz_{k+1},\dots,iz_n$ and $z_{n-k+1},\dots,z_n$. We order these elements in
such a way that we get a simple expression for $\langle\al,\be\rangle$ and thus use
  $$
\al=iz_{n-k+1}\wedge\dots\wedge iz_n\wedge
  iz_{k+1}\wedge\dots\wedge iz_{n-k}\wedge z_{n-k+1}\wedge\dots\wedge
  z_n.
$$
Then $\langle\al,\be\rangle$ is the determinant of the $(n\x n)$--matrix of mutual
inner products between the factors in the wedge product. We have arranged things in
such a way that this matrix is diagonal with $s$ entries equal to $-1$ and $n-s$
entries equal to $1$, so $\langle\al,\be\rangle=(-1)^s$. On the other hand,
reordering the wedge products, we conclude that $\al\wedge\be=(-1)^n\vol$ and thus
$*\be=(-1)^{n-s}\be$.
\end{proof}

This implies a complete description of the set of $H$--orbits in the isotropic
Grassmannian $IGr^\pm(n,\Bbb R^{(n,n)})$, where the superscript indicates self--duality
respectively anti--self--duality. We denote by $\Cal O_{(r,s)}$ the set of those
maximal isotropic subspaces on which the restriction of the real part of $b$ has
signature $(r,s)$. Then we get $IGr^+(n,\Bbb R^{(n,n)})$ is the union of the
$\Cal O_{(r,s)}$ with $r+s\leq n$, such that both $n-r-s$ and $n-s$ are even even,
while $IGr^-(n,\Bbb R^{(n,n)})$ is the union of the orbits for which $n-r-s$ is even but
$n-s$ is odd. The open orbits are exactly those in which $r+s=n$, while
$\Cal O_{(r',s')}\subset\overline{\Cal O}_{(r,s)}$ if and only if $r'\leq r$ and
$s'\leq s$ and $s'$ has the same parity as $s$.

In particular, there is the orbit $\Cal O_{(n,0)}\subset IGr^+(n,\Bbb R^{(n,n)})$ which,
for the standard complex bilinear form $b$, contains the subspace
$\Bbb R^n\subset\Bbb C^n$. The stabilizer of this subspace in $H$ visibly is given by
the matrices in $SO(n,\Bbb C)$ for which all entries are real, so this is just
$SO(n)$. Thus, we get a parabolic compactification of $H/K:=SO(n,\Bbb C)/SO(n)$ of
the form $\overline{H/K}=\cup_{i=0}^{\lfloor n/2\rfloor}\Cal
O_{(n-2i,0)}$.
Similarly, one obtains parabolic compactifications of $SO(n,\Bbb C)/SO_0(p,q)$ for
$p+q=n$.

From the proof of Lemma \ref{lem4.1}, one can also see that already for the isotropic
Grassmannians $IGr(\ell,\Bbb R^{(n,n)})$ with $\ell<n$, the orbits of
$H=SO(n,\Bbb C)$ are not determined by rank and signature of the restriction of the
real part of $b$ alone. Indeed, for an isotropic subspace $V\subset\Bbb R^{(n,n)}$ of
dimension $\ell<n$, the null space of $\Re(b)$ is $J(V)\cap V^\perp$. Hence the
co--rank does not have to be even in general, and there is an additional
distinguished subspace in the null space of $\Re(b)$, namely the maximal complex
subspace $J(V)\cap V$ of $V$. Hence the real dimension of this subspace, which has to
be even and at most equal to the co--rank, is an additional invariant preserved by
the action of $H$. So the two--step flag $V\subset V^\perp$ plays a similar role as
in the discussion in Section \ref{3.1}. Passing to more general isotropic flag
manifolds, one gets additional invariants depending on the relative position of the
individual constituents of the flag and their orthocomplements with respect to $J$.

Let us next pass to the description of the individual orbits and of the infinitesimal
transversal. Similar to the discussion in Proposition \ref{3.1} each orbit admits a
fibration onto a generalized flag manifold of $H$ with fiber a lower dimensional
analog of $H/K$. The infinitesimal transversals are described by Hermitian matrices,
which is surprising, since the setup does not seem to naturally lead to a Hermitian
inner product.
\begin{prop}\label{prop4.1}
  Fix $r$ and $s$ such that $n-r-s=2k$. Then we have

  (1) The orbit $\Cal O:=\Cal O_{(r,s)}\subset IGr^\pm(n,\Bbb R^{(n,n)})$ has
  codimension $k^2$. There is an $H$--equivariant, surjective submersion from
  $\Cal O_{(r,s)}$ onto the Grassmannian $IGr(k,\Bbb C^n)$ of $k$--dimensional
  complex subspaces which are isotropic for $b$. The fibers of this map are
  isomorphic to $SO(n-2k,\Bbb C)/SO(r,s)$.

  (2) Replacing $H$ by a conjugate subgroup $H_{\Cal O}$, as described
  in Section \ref{2.4}, the infinitesimal transversal of $\Cal O$ is
  given by the natural representation of $GL(k,\Bbb C)$ on the space
  of Hermitian $k\x k$--matrices via a natural surjective homomorphism
  $H_{\Cal O}\cap P\to GL(k,\Bbb C)$.
\end{prop}
\begin{proof}
  (1) We continue using the notation for orthocomplements from the
  proof of Lemma \ref{lem4.1}. From there we know that the complex
  subspace $W:=J(V)\cap V$ has complex dimension $k$ and is isotropic
  for $b$. Sending $V$ to $W$ defines the claimed map to the complex
  isotropic Grassmannian. We have also seen in that proof that $V$
  descends to a totally real subspace $\underline{V}\subset
  W^{\perp_b}/W$ which is isotropic for the imaginary part of the
  induced complex bilinear form $\underline{b}$, while the restriction
  of the real part of $\underline{b}$ to $\underline{V}$ is
  non--degenerate of signature $(r,s)$. This leads to the description
  of the fiber. Since conversely starting from a complex isotropic
  subspace $W\subset\Bbb C^n$ and an appropriate totally real subspace
  $\underline{V}\subset W^{\perp_b}/W$, the pre--image of
  $\underline{V}$ in $\Bbb C^n$ evidently lies in $\Cal O$, we get
  surjectivity. The codimension of the orbit follows from standard
  results on the dimensions of the spaces involved. Part (2) provides
  an alternative proof.

  (2) Passing to $H_{\Cal O}$ simply means that we use matrix representations with
  respect to a basis adapted as to $V\in\Cal O$ as in Lemma \ref{lem4.1}, and then
  use the stabilizer of $V$ in $G$ as our parabolic subgroup $P$. We use a basis as
  obtained in Lemma \ref{lem4.1} to split real linear endomorphisms of
  $\Bbb C^n\cong\Bbb R^{2n}$ into blocks in two different ways. In the notation of
  that lemma, we order the real basis vectors by splitting them into $6$ groups of
  vectors. We first take $z_1,\dots,z_k$, then $iz_1,\dots,iz_k$, next
  $z_{k+1},\dots, z_{n-k}$, then $iz_{k+1},\dots, iz_{n-k}$, then
  $z_{n-k+1},\dots,z_n$ and finally $iz_{n-k+1},\dots,iz_n$. Correspondingly, we
  write elements of $\frak g$ as block matrices with $6\x 6$ blocks of sizes $k$,
  $k$, $n-2k$, $n-2k$, $k$, and $k$, respectively. We will use the notation $(M_{jk})$
  for this block decomposition, with $j,k=1,\dots,6$. From the construction, it is
  clear how $J$ is described in terms of such block matrices.

  On the other hand, we can use the coarser decomposition into four blocks of size
  $n\x n$, by collecting the first three groups of basis vectors and the last three
  groups of basis vectors. Here the first $n$ basis vectors by construction form a
  basis of $V$ and the span of the last $n$ basis vectors is also isotropic for
  $\langle\ , \rangle$. Correspondingly, the matrix of $\langle\ , \rangle$ in that
  block decomposition has the form $\begin{pmatrix} 0 & \Bbb J^t \\ \Bbb J & 0
  \end{pmatrix}$. Here $\Bbb J$ is the $n\x n$--matrix, which split into blocks of
  sizes $k$, $k$ and $n-2k$ has the form $\left(\begin{smallmatrix} 0 & 0 & \Bbb I_{r,s}\\
    0 & \Bbb I & 0\\ \Bbb I& 0 & 0\end{smallmatrix}\right)$,
  with notation as in Lemma \ref{lem4.1}.  Correspondingly, the matrices in $\frak g$
  are exactly those which have the coarse block form
  $\begin{pmatrix} A_{11} & A_{12} \\ A_{21} & -\Bbb JA_{11}^t\Bbb J^t\end{pmatrix}$
  with $n\x n$--matrices $A_{ij}$ such that $A_{12}=-\Bbb JA_{12}^t\Bbb J^t$ and
  $A_{21}=-\Bbb JA_{21}^t\Bbb J^t$.

  Now expressing the condition on $A_{21}$ in terms of the finer block decomposition
  immediately shows that this has to have the form $
  \begin{pmatrix} M_{41} & M_{42} & M_{43} \\ M_{51} & M_{52} & -M_{42}^t\Bbb I_{r,s}
    \\ M_{61} & -M_{51}^t & -M_{41}^t \Bbb I_{r,s} \end{pmatrix}$
  with $M_{52}^t=-M_{52}$ $M_{61}^t=-M_{61}$ and $M_{43}\in\frak o(r,s)$. From the
  explicit form of $J$, it is easy to describe the subalgebra
  $\frak h\subset\frak g$. It simply consists of those block matrices in $\frak g$
  such that for each odd $j$ and $k$ we have $M_{j+1,k+1}=M_{j,k}$ and
  $M_{j,k+1}=-M_{j+1,k}$. This confirms that any element of $\frak h\cap\frak p$ also
  stabilizes $W$ and $W^{\perp_b}$ (which are spanned by the first $2k$ respectively
  all but the last $2k$ basis vectors). Moreover, it shows that the action of
  $H\cap P$ on $\Bbb C^n/W^{\perp_b}$ defines a surjective homomorphism
  $H\cap P\to GL(k,\Bbb C)$. On the Lie algebra level, this is represented by the
  block $\begin{pmatrix} M_{55} & M_{56} \\ M_{65} & M_{66}\end{pmatrix}$ (for which
  we know that $M_{56}=-M_{65}$ and $M_{66}=M_{55}$. We also know from above that we
  get $M_{11}=M_{22}=-M_{55}^t$ while $M_{21}=-M_{12}=-M_{65}^t$.

  For the block decomposition of $A_{21}$ from above, we get
  additional restrictions, namely $M_{51}=-M_{51}^t$ and
  $M_{52}=-M_{61}$. A short computation shows that one does not obtain
  further restrictions on $M_{41}$, $M_{42}$, and $M_{43}$, they only
  have to be related to blocks in the row above. From this, we can
  read off a complement to $\frak h/(\frak h\cap\frak p)$ in $\frak
  g/\frak p$. The matrices whose only nonzero blocks are $M_{51}$,
  $M_{52}$, $M_{61}=M_{52}$ and $M_{62}=-M_{51}$ such that
  $M_{51}^t=M_{51}$ and $M_{52}^t=-M_{52}$ visibly descend to such a
  complement. Since $M_{51}$ and $M_{61}$ both are $k\x k$--matrices
  we see that $\Cal O$ indeed has codimension $k^2$.

  To determine the action of $\frak h\cap\frak p$ on the infinitesimal transversal, we
  can compute the adjoint action on this complementary subspace and then project to
  the quotient. For $(N_{jk})\in\frak h\cap\frak p$, the resulting action depends
  only on the blocks $N_{jk}$ for $j,k=1,2$ (which also determine the parts for
  $j,k=5,6$). Indeed, a short computation shows that the action sends $
  \begin{pmatrix} M_{51} & M_{52} \\ M_{52} & -M_{51}\end{pmatrix}$ (where we still
  have $M_{51}^t=M_{51}$ and $M_{52}^t=-M_{52}$ to a matrix of the same form with
  first column given by 
$$
\begin{pmatrix}
  -N_{11}^tM_{51}-N_{12}^tM_{61}-M_{51}N_{11}+M_{61}N_{12} 
% &   -N_{11}^tM_{61}+N_{12}^tM_{51}-M_{51}N_{12}-M_{61}N_{11} 
  \\
  N_{12}^tM_{51}-N_{11}^tM_{61}-M_{61}N_{11}-M_{51}N_{12} 
% &  N_{12}^tM_{61}+N_{11}^tM_{51}-M_{61}N_{12}+M_{51}N_{11}
\end{pmatrix}.
$$ 
This shows that our complementary subspace  is invariant under the
action. Moreover, then considering the Hermitian matrix $M_{51}+iM_{61}$ this action
exactly corresponds to the standard action
$$
(-N_{11}^t+iN_{12}^t)(M_{51}+iM_{61})-(M_{51}+iM_{61})(-N_{11}-iN_{12})
$$ 
of $\frak{gl}(k,\Bbb C)$ on Hermitian $k\x k$--matrices.
\end{proof}

\subsection{Orbit closures via BGG solutions}\label{4.2} 
The first steps of this are closely similar to the case discussed in Section
\ref{3.2}. We consider one of the isotropic Grassmannians $IGr^\pm(n,\Bbb
R^{(n,n)})=G/P$, where $G=SO_0(n,n)$. The tautological bundles on these Grassmannians
are determined by the subbundle $E$ in the trivial bundle with fiber $\Bbb R^{(n,n)}$,
whose fiber at a point $V\in IGr^\pm(n,\Bbb R^{(n,n)})$ is the subspace $V\subset\Bbb
R^{(n,n)}$. Via the invariant inner product $\langle\ ,\ \rangle$ on $\Bbb R^{2n}$, the
quotient bundle $\Bbb R^{(n,n)}/E$ gets identified with $E^*$. For our choice of
parabolic subgroup $P$, the Levi factor $G_0\subset P$ is $GL(n,\Bbb R)$ and the
bundle $E$ is the completely reducible bundle associated to the standard
representation of that group. From the description in the proof of Proposition
\ref{prop4.1} it is clear that $\frak g/\frak p$ is isomorphic to $\La^2\Bbb R^{n*}$
as a representation of $G_0$, while the nilradical of $\frak p$ acts trivially on
$\frak g/\frak p$. Correspondingly, the tangent bundle $T(G/P)$ is naturally
identified with $\La^2E^*$, which defines an \textit{almost spinorial structure} on
$G/P$, making it into the homogeneous model for such structures. 

Completely reducible natural bundles on $G/P$ are induced by representations of the
group $G_0=GL(n,\Bbb R)$, so they can be built up from the bundles $E$ and $E^*$ by
tensorial constructions. In particular, for $k=1,\dots,n$ we can form $\La^kE^*$ and
then take the subbundle $\ocirc^2(\La^kE^*)\subset S^2(\La^kE^*)$ as described in
Section \ref{3.2}. By construction, this is an irreducible natural bundle over $G/P$,
and for $k=1$, $n-1$, and $n$ it coincides with $S^2(\La^kE^*)$. In particular, it is
a density bundle for $k=n$.  The simplest tractor bundle for almost spinorial
structures is the standard tractor bundle $\Cal T$, which corresponds to the standard
representation $\Bbb R^{(n,n)}$ of $G$. On the homogeneous model $G/P$, this bundle is
naturally isomorphic to the trivial bundle $(G/P)\x\Bbb R^{(n,n)}$. The $P$--invariant
filtration of $\Cal T$ just consists of the subbundle $E$ for which $\Cal T/E\cong
E^*$. More generally, any $G$--irreducible subrepresentation of a tensor power of
$\Bbb R^{(n,n)}$ induces a tractor bundle over $G/P$.

\begin{thm}\label{thm4.2}
Consider the decomposition of one of the isotropic Grassmannians $IGr^\pm(n,\Bbb
R^{(n,n)})$ into the orbits $\Cal O_{(r,s)}$ for the subgroup $H\subset G$ as described
in Proposition \ref{prop4.1}. For each $k=1,\dots,n$, there is a section
$\si_k\in\Ga(\ocirc^2(\La^kE))$, which lies in the kernel of the first BGG operator
defined on that bundle whose zero--locus is the union of those $H$--orbits $\Cal
O_{(r,s)}$, for which $r+s<k$. The relevant first BGG operator is of first order if
$k<n-1$, of second order for $k=n-1$ and of third order for $k=n$.
\end{thm}
\begin{proof}
  The basic strategy is the same as in the proof of Theorem \ref{thm3.2}, but some of
  the representation theory is more involved. We continue viewing $\Bbb R^{2n}$ as
  $\Bbb C^n$ and $\langle\ ,\ \rangle$ as the imaginary part of the standard
  symmetric complex bilinear form $b$. Then multiplication by $i$ defines an
  endomorphism of $\Bbb R^{(n,n)}$ which is symmetric for $\langle\ ,\ \rangle$ and
  trace--free. Thus it defines an element $J\in S^2_0\Bbb R^{(n,n)*}$ whose stabilizer
  in $G$ is the subgroup $H$. This in turn defines a parallel section $s_1$ of the
  tractor bundle $S^2_0(\Cal T^*)$ which can be either considered as an endomorphism
  of $\Cal T$ or as a symmetric bilinear form on that bundle. The irreducible qoutient
  of $S^2_0\Cal T^*$ is just $S^2E^*$ and projecting $s_1$, we obtain a section
  $\si_1\in\Ga(S^2E^*)$, which lies in the kernel of the first BGG operator defined
  on that bundle. For a point $V\in G/P$, $E_V=V$ and $\si_1(V)$ by construction is
  the restriction of $\Re(b)$ to $V$. Thus $V$ lies in $\Cal O_{(r,s)}$ if and only
  if $\si_1(V)$ has signature $(r,s)$ (and thus rank $r+s$).

  Parallel to the discussion in Section \ref{3.2}, for $k=1,\dots,n$, the
  endomorphism $\La^kJ$ of $\La^k\Bbb R^{(n,n)}$ can also be viewed as a bilinear form
  on $\La^k\Bbb R^{(n,n)}$. By Lemma \ref{lem3.2} this sits in the
  $GL(2n,\Bbb R)$--irreducible component  $W_k\subset S^2(\La^k\Bbb R^{2n*})$ of
  maximal highest weight. Now for most $k$, $W_k$ is not irreducible for $G$, since
  the inner product $\langle\ ,\ \rangle$ defines non--trivial traces on it. The
  joint kernel of these traces is the $G$--irreducible component of highest weight in
  $S^2(\La^k\Bbb R^{(n,n)*})$, which we denote by $\ocirc_0^2(\La^k\Bbb R^{(n,n)*})$. It
  turns out that, as a representation of $G$, $W_k$ is isomorphic to
  $\oplus_{j=0}^k\ocirc_0^2(\La^j\Bbb R^{(n,n)*})$. Here for $j=0,1$, we obtain a
  trivial summand and $S^2_0(\Bbb R^{(n,n)})$, respectively. Now we can split
  $\La^kJ\in W_k$ according to this decomposition. The component of $\La^kJ$ in
  $\ocirc^2_0(\La^k\Bbb R^{(n,n)*})$ defines a section
  $s_k\in\Ga(\ocirc^2_0(\La^k\Cal T^*))$, which is parallel for the canonical tractor
  connection. Projecting to the irreducible quotient bundle, we obtain a section
  $\si_k\in\Ga(\ocirc^2(\La^kE^*))$ which lies in the kernel of the first BGG
  operator defined on that bundle.

  Returning to the decomposition of $W_k$ into irreducibles, we can split each of the
  representations $\ocirc_0^2(\La^j\Bbb R^{(n,n)*})$ into irreducibles with respect to
  $G_0=GL(n,\Bbb R)$. In particular, the $P$--irreducible quotient of
  $\ocirc_0^2(\La^j\Bbb R^{(n,n)*})$ is $\ocirc^2(\La^j\Bbb R^{n*})$, which easily
  implies that for $j<k$, there is no nonzero $G_0$--equivariant map
  $\ocirc_0^2(\La^j\Bbb R^{(n,n)*})\to\ocirc^2(\La^k\Bbb R^{n*})$. But this implies that
  the natural projection from $W_k$ to its irreducible quotient
  $\ocirc^2(\La^k\Bbb R^{n*})$ which we used in the proof of Theorem \ref{thm3.2}
  factors through $\ocirc^2_0(\La^k\Bbb R^{(n,n)*})$. But this shows that $\si_k$ is
  induced by the image of $\La^kJ$ under the latter projection, so from the proof of
  Theorem \ref{thm3.2} we see that $\si_k=\La^k\si_1$, which leads to the description
  of the zero loci, as in that proof. The orders of the BGG operators can again be
  read off from the weights of the inducing representations.
\end{proof}

\subsection{A slice theorem}\label{4.3}
In Proposition \ref{prop4.1}, we have obtained infinitesimal transversals which are
formed by Hermitian matrices. This is rather surprising, since initially there does
not seem to be a natural notion of conjugation around. We start by proving a Lemma
which directly explains how Hermitian metrics enter the picture.
\begin{lemma}\label{lem4.3}
  For even $k=2\ell$, consider the standard complex bilinear form $b$ on $\Bbb C^k$
  and let $\langle\ ,\ \rangle$ be its imaginary part. Let $V\subset\Bbb C^k$ be a
  real subspace of dimension $k$, which is isotropic for $\langle\ ,\ \rangle$.
  Suppose that $\Bbb C^k=Z_1\oplus Z_2$ is a decomposition into a sum of two complex
  subspaces, both of which are isotropic for $b$, and such that
  $Z_2\cap V=\{0\}$, and hence the projection $\pi$ onto the first summand
  restricts to a (real) linear isomorphism $V\to Z_1$. Then the restriction of the
  real part $\Re(b)$ of $b$ to $V$ is Hermitian with respect to the pullback along
  $\pi|_V$ of the complex structure on $Z_1$.
\end{lemma}
\begin{proof}
  Composing the projection onto the second summand with $(\pi|_V)^{-1}$, we obtain a
  real linear map $\ph:Z_1\to Z_2$ such that $V=\{z+\ph(z):z\in Z_1\}$. For $j=1,2$
  take $v_j\in V$ and write it as $v_j=z_j+\ph(z_j)$ for $z_j\in Z_1$ to obtain
  $b(v_1,v_2)=b(z_1,\ph(z_2))+b(\ph(z_1),z_2)$. Since $V$ is isotropic for
  $\langle\ ,\ \rangle$, we conclude that
  $\langle z_1,\ph(z_2)\rangle=-\langle \ph(z_1),z_2\rangle$. Since $z_1,z_2\in Z_1$
  may be arbitrary, $\ph$ is skew symmetric with respect to $\langle\ ,\ \rangle$. On
  the other hand, we can compute $\Re(b)(v_1,v_2)$ as the real part of the above
  expression, which coincides with
  $\langle iz_1,\ph(z_2)\rangle+\langle \ph(z_1),iz_2\rangle$. Skew symmetry of $\ph$
  readily implies that this remains unchanged if we replace $z_1$ by $iz_1$ and $z_2$
  by $iz_2$, which implies the claim of the lemma.
\end{proof}

Having this at hand, we can prove the slice theorem. Given an orbit $\Cal
O_{(r,s)}\subset IGr^\pm(n,\Bbb R^{(n,n)})$, we know from Proposition \ref{prop4.1}
that $n-r-s$ is even, and we denote this by $2k$. Now we denote by $\Cal H_k$ the set
of real matrices of size $2k\x 2k$ which represent complex matrices of size $k\x k$
that are Hermitian. In particular, these are real symmetric matrices of even rank for
which both parts of the signature are even.

\begin{thm}\label{thm4.3}
  Each of the orbits $\Cal O_{(r,s)}\subset G/P:=IGr^\pm(n,\Bbb R^{(n,n)})$ is an
  embedded submanifold. For each point $x\in\Cal O_{(r,s)}$, there is an open
  neighborhood $U$ of $x$ in $G/P$ and a diffeomorphism $\ph$ from $U$ onto an open
  neighborhood of $(U\cap\Cal O_{(r,s)})\x\{0\}$ in $(U\cap\Cal O_{(r,s)})\x\Cal H_k$
  such that $U\subset \cup_{r'\geq r,s'\geq s}\Cal O_{(r',s')}$ and a point $y\in U$
  lies in $\Cal O_{(r',s')}$ if and only if the second component of $\ph(y)$ has
  signature $(r'-r,s'-s)$.
\end{thm}
\begin{proof}
  The first part of this is closely parallel to the proof of Theorem \ref{thm3.3}:
  There is a connected neighborhood $\underline{U}$ of $x$ in $\Cal O_{(r,s)}$, which
  is an embedded submanifold of $G/P$. Possibly shrinking $\underline{U}$, there is a
  connected open neighborhood $U$ of $x\in G/P$ and a smooth subbundle
  $W\subset E|_U$ of rank $r+s$ such that for each $y\in U$, the restriction of the
  symmetric bilinear form $\si_1(y)$ to $W_y\subset E_y$ is non--degenerate and has
  signature $(r,s)$. Defining $\tilde E_y\subset E_y$ to be the space of those
  $v\in E_y$ such that $\si_1(y)(v,w)=0$ for all $w\in W_y$, we obtain a smooth
  subbundle $\tilde E\subset E|_U$ of even rank $2\ell$. By construction, for
  $y\in U$ the null--space of $\si_1(y)$ is contained in $\tilde E_y$ and it
  coincides with $\tilde E_y$ if $y\in\Cal O_{(r,s)}$. In particular, we conclude
  that $U\subset \cup_{r'\geq r,s'\geq s}\Cal O_{(r',s')}$ and that a point $y\in U$
  lies in $\Cal O_{(r',s')}$ if and only if the restriction of $\si_1(y)$ to
  $\tilde E_y$ has signature $(r'-r,s'-s)$.

  At this point we need additional input to get the Hermitian aspects into the
  picture. For each $y\in U$, $W_y$ is a real subspace in $\Bbb C^n$ of real
  dimension $r+s$ on which $\langle\ ,\ \rangle$ is identically zero. But $\Re(b)$ is
  non--degenerate on $W_y$, so $W_y\cap J(W_y)=\{0\}$ and that the restriction of $b$
  to the complex subspace $W_y\oplus J(W_y)\subset \Bbb C^n$ is
  non--degenerate. Otherwise put, viewing $W$ as a smooth subbundle in the standard
  tractor bundle $\Cal T|_U$ and viewing the parallel section $s_1$, from the proof of
  Theorem \ref{thm4.2}, as a section of $L(\Cal T,\Cal T)$, we can form a smooth
  subbundle $W\oplus s_1(W)\subset\Cal T$, which by construction is invariant under
  $s_1$. Using $s_1$, we can extend the canonical inner product $\langle\ ,\ \rangle$
  on $\Cal T$ to a field of non--degenerate $\Bbb C$--valued bilinear forms, which are
  complex bilinear with respect to $s_1$. The restriction of this form to
  $W\oplus s_1(W)$ is non--degenerate, so we can form the complex orthocomplement,
  which is an $s_1$--invariant subbundle $Z\subset\Cal T$ on which the complex bundle
  metric is non--degenerate, too.

  Now by definition, $\tilde E_y$ is perpendicular to $W_y$ with respect to $\Re(b)$
  and since $\tilde E_y\subset E_y$, it is also perpendicular to $W_y$ with respect
  to $\langle\ ,\ \rangle$. This implies that $\tilde E_y\subset Z_y$ for all
  $y\in U$. On the other hand, the complex corank of $Z$ in $\Cal T$ by construction
  equals the real corank of $W$ in $E$ and thus is even. This implies that locally we
  can write $Z=Z_1\oplus Z_2$ for two smooth subbundles which are invariant under
  $s_1$ and isotropic for the complex bundle metric. Possibly shrinking $U$ and
  starting in such a way that in the point $x\in\Cal O_{(r,s)}$ we take $Z_1$ to be
  $\tilde E_x$ (which coincides with the null space of $\si_1(x)$ and thus is complex
  isotropic), we may assume that $Z_2\cap\tilde E=\{0\}$ on all of $U$. But this
  implies that the projection onto the first factor restricts to an isomorphism
  $\tilde E\to Z_1$ of real vector bundles. Pulling back the complex structure from
  $Z_1$ then makes $\tilde E$ into a complex vector bundle of complex rank $k$, and
  applying Lemma \ref{lem4.3} point--wise, we conclude that the restriction of
  $\si_1$ to $\tilde E$ defines a Hermitian bundle metric on the complex vector
  bundle $\tilde E$.

  Having this at hand, we can continue as in the proof of Theorem \ref{thm3.3}.  and
  Corollary \ref{cor3.3}. Restricting bilinear forms defines a map
  $q:S^2E^*\to S^2\tilde E^*$ and we know that $q\o\si_1$ has values in the subbundle
  $\Cal H^2\tilde E^*$ of forms which are Hermitian with respect to the complex
  structure constructed above. We also know that $q\o \si_1$ vanishes identically
  along $\underline{U}:=U\cap \mathcal{O}_{(r,s)}\subset U$. To complete the proof, it suffices to show that,
  possibly shrinking $U$ and $\underline{U}$, $q\o\si_1\in\Ga(\Cal H^2\tilde E^*)$ is
  a defining section for $\underline{U}$. Since $q\o\si_1$ by construction vanishes
  along $\Cal O_{(r,s)}$ this boils down to proving that
  $\nabla(q\o\si)(y):T_y(G/P)\to \Cal H^2\tilde E^*$ is surjective for each
  $y\in\Cal O_{(r,s)}$. Having shown that, we can complete the proof exactly as the
  one of Corollary \ref{cor3.3}.

  The first steps in the analysis of the derivative are as in the proof of Theorem
  \ref{thm3.3} and we use the notation from that proof. Any linear connection on $E$
  induces a linear connection on $\tilde E$, there are induced linear connections
  $\nabla$ on $S^2E$ and $\tilde\nabla$ on $S^2\tilde E$ and for each vector field
  $\xi$, we get $\tilde\nabla_\xi(q\o\si)(y)=q(\nabla_\xi\si(y))$. Starting with a
  Weyl connection on $E$, we can use Proposition \ref{prop2.7} to compute
  $\nabla\si(y)$ as $-\partial(\mu(y))$. For the isotropic Grassmannian, the tangent
  bundle $T(G/P)$ is isomorphic to $\La^2E^*$ (which defines the flat almost
  spinorial structure) while for $\Cal V=S^2_0\Cal T^*$, we get
  $\Cal V/\Cal V^1=S^2E^*$ and $\Cal V^1/\Cal V^2=(E^*\otimes E)_0$, where the
  subscript indicates the trace--free part. Hence $\partial$ maps $(E^*\otimes E)_0$
  to $\La^2E\otimes S^2E^*$ and since this comes from a $GL(n,\Bbb R)$--equivariant
  map between the inducing representations, it has to be given by tensoring with the
  identity and then symmetrizing the $E$--components and alternating the
  $E^*$--components. Otherwise put, if we view $\mu$ as a linear map $E^*\to E^*$,
  then $\partial(\mu)$ sends $\al\wedge\be$ to (a nonzero multiple of)
  $\mu(\al)\vee\be-\mu(\be)\vee\al$.

  Now for $y\in\Cal O_{(r,s)}$ corresponding to the isotropic subspace $E_y$, we know
  that $\tilde E_y$ is the null--space of $\si_1(y)$ and from the proof of Lemma
  \ref{lem4.1} we know that this coincides with the maximal complex subspace of
  $E_y$. Hence we get $s_1(y)(E_y)\cap E_y=\tilde E_y$ and $s_1(y)$ makes
  $\tilde E_y$ into a complex vector space. The fact that, viewed as an endomorphism
  of $\Cal T$, $s_1(y)$ maps $\tilde E_y$ to $\tilde E_y\subset E_y$ says that
  $\si_1(y)|_{E_y}=0$ and hence $s_1(y)|_{\tilde E_y}=\mu(y)|_{\tilde E_y}$. Thus the
  restriction of $\mu(y)$ to $\tilde E_y$ is simply multiplication by $i$ for the
  complex structure on $\tilde E_y$ we have constructed. By definition of the
  structure on the dual, $\mu(y)$ is given by multiplication by $i$ as a map
  $\tilde E_y^*\to\tilde E_y^*$. But this exactly means that for
  $\al,\be\in\tilde E_y^*$, we get (up to a non--zero factor)
  $\partial(\mu)(\al\wedge\be)=i\al\vee\be-\al\vee i\be$, so this lies in
  $\Cal H^2\tilde E_y^*$. On the other hand, any element of $\Cal H^2\tilde E_y^*$
  can be written as a linear combination of elements of the form
  $\al\vee\be+i\al\vee i\be$ and such an element is obtained (up to a factor) as
  $-\partial(\mu)(i\al\wedge\be)$, which completes the proof.
\end{proof}

\begin{remark}\label{rem4.3}
  (1) Similarly to Remark \ref{rem3.4}, there is a nice interpretation of Theorem
  \ref{thm4.3} in terms of local compactifications. Let us first consider the case
  that $H/K=SO(n,\Bbb C)/SO(n)\cong\Cal O_{(n,0)}$ for which we know that
  $\overline{H/K}=\cup_\nu\Cal O_{(n-2\nu,0)}$ with integers $\nu$ such that
  $0\leq\nu\leq \tfrac{n}{2}$. Then the model is a local compactification of
  $GL(\nu,\Bbb C)/U(\nu)$, which can be identified with the space positive definite
  Hermitian $\nu\x\nu$--matrices, locally compactified by the space of positive
  semi--definite Hermitian matrices. Theorem \ref{thm4.3} then says that for
  sufficiently small open subsets $W\subset \Cal O_{(n-2\nu,0)}$, a neighborhood of
  $W$ in $\overline{H/K}$ is isomorphic to the product of $W$ with a neighborhood of
  $0$ in that local compactification of $GL(\nu,\Bbb C)/U(\nu)$.

  For other signatures, i.e.~$H/K=SO(n,\Bbb C)/SO(p,q)\cong \Cal O_{(p,q)}$ with
  $p+q=n$, there is a similar description in terms of local compactifications of
  $GL(\nu,\Bbb R)/U(p',q')$ for $p'+q'=\nu$, $p'\leq p$ and $q'\leq q$. The relevant
  local compactification of Hermitian matrices of signature $(p',q')$ is given as
  Hermitian matrices of signatures $(r,s)$ with $r\leq p'$ and $s\leq q'$.
  
(2) As we have seen already in Lemma \ref{lem4.1} the ranks always drop in steps of
  two in our example. This did not cause problems in the description of orbit
  closures via zero--loci of BGG solutions in Theorem \ref{thm4.2}. However, already
  there the strange situation occurs that, for example, the sections $\si_n$ and
  $\si_{n-1}$ have the same zero locus, since rank smaller than $n$ always implies
  rank smaller than $n-1$. The problem becomes serious, however, when one tries to
  generalize Proposition \ref{prop3.4} to the current setting. While the orbits $\Cal
  O_{(r,s)}$ of largest non--full rank (i.e.~with $r+s=n-2$) still form embedded
  hypersurfaces in $G/P$ by Theorem \ref{thm4.3}, the section $\si_n$ (which is the
  only $\si_i$ that has values in a line bundle), certainly is not a defining density
  for these hypersurfaces. Indeed, arguments similar to the ones used in the proof of
  Proposition \ref{prop3.4} show that $\si_n$ and $\nabla\si_n$ simultaneously vanish
  in points where $\si_1$ has rank less than $n-1$, so this happens in all points of
  $\Cal O_{(r,s)}$.

Indeed, in the current situation, we do not see a way how to construct a natural
defining density for the orbits which are hypersurfaces. To obtain a defining
density, it would seem necessary to exploit the fact that, as shown in the proof of
Theorem \ref{thm4.3}, the metric on a two--dimensional complement to the tangent
space of the orbit is Hermitian for an appropriate complex structure. However, that
complex structure seems to be canonical only along the orbit itself, where this
transverse metric vanishes identically. 
\end{remark}

\end{document}